\documentclass{amsart}

\usepackage{amsmath,amssymb}
\usepackage{amsthm}
\usepackage{stmaryrd}
\usepackage{enumerate}
\usepackage{url}
\usepackage{wasysym}
\usepackage{color}
\usepackage[normalem]{ulem}
\usepackage[hidelinks]{hyperref}

\usepackage{graphicx}

\makeatletter
\providecommand{\bigsqcap}{%
  \mathop{%
    \mathpalette\@updown\bigsqcup
  }%
}
\newcommand*{\@updown}[2]{%
  \rotatebox[origin=c]{180}{$\m@th#1#2$}%
}
\makeatother

\newtheorem{theorem}{Theorem}[section]
\newtheorem{lemma}[theorem]{Lemma}
\newtheorem{corollary}[theorem]{Corollary}
\newtheorem{proposition}[theorem]{Proposition}
\newtheorem{fact}[theorem]{Fact}

\theoremstyle{definition}
\newtheorem{definition}[theorem]{Definition}
\newtheorem{example}[theorem]{Example}
\newtheorem{remark}[theorem]{Remark}
\newtheorem{convention}[theorem]{Convention}

\newtheorem{question}[theorem]{Question}

\def\tp{\operatorname{tp}}

\def\supp{\operatorname{supp}}

\def\inv{\operatorname{inv}}

\def\amal{\operatorname{Amal}}
\def\am{\operatorname{Am}}
\def\E{\operatorname{E}}
\def\Aut{\operatorname{Aut}}

\def\Ind{\setbox0=\hbox{$x$}\kern\wd0\hbox to 0pt{\hss$\mid$\hss}
\lower.9\ht0\hbox to 0pt{\hss$\smile$\hss}\kern\wd0}

\def\Notind{\setbox0=\hbox{$x$}\kern\wd0\hbox to 0pt{\mathchardef
\nn=12854\hss$\nn$\kern1.4\wd0\hss}\hbox to
0pt{\hss$\mid$\hss}\lower.9\ht0 \hbox to 0pt{\hss$\smile$\hss}\kern\wd0}

\title{An Invitation to Extension Domination}
\author[K. Gannon]{Kyle Gannon}

\author[J. Ye]{Jinhe Ye}

\date{\today}

\address{Beijing International Center for Mathematical Research\\
Peking University\\
Beijing, China}
\email{kgannon@bicmr.pku.edu.cn}

\address{Mathematical Institute\\University of Oxford\\
Oxford, OX2 6GG, UK}

\email{jinhe.ye@maths.ox.ac.uk}
\thanks{The second author was partially supported by GeoMod AAPG2019 (ANR-DFG), Geometric and
  Combinatorial Configurations in Model Theory. Both authors would like to express their gratitude towards Rosario Mennuni for his helpful comments.}
\begin{document}

\begin{abstract} Motivated by the theory of domination for types, we introduce a notion of domination for Keisler measures called extension domination. We argue that this variant of domination behaves similarly to its type setting counterpart. We prove that extension domination extends domination for types and that it forms a preorder on the space of global Keisler measures. We then explore some basic properties related to this notion (e.g. approximations by formulas, closure under localizations, convex combinations). We also prove a few preservation theorems and provide some explicit examples. 
\end{abstract}

\maketitle
\vspace{-15pt} 
\section{Introduction} 
The notion of domination was introduced by Lascar in \cite{lascar1982ordre} in the context of stable theories. This notion again appears in \cite{lascar1990stability} which has been studied in more detail by Pillay~\cite{GST} . 
Later, it was studied in the general setting of complete first order theories ~\cite{rosario-inv-dom} and extensively studied in the specific context of valued fields ~\cite{HHM-stable-domination}. Loosely speaking, \textit{$p(x)$ dominates $q(y)$} captures the idea that $p$, along with a small amount of auxiliary \textit{$xy$-data}, controls $q$. This relation forms a preorder on the space of global types. In practice, sometimes one can understand the space of invariant types by studying a collection of \textit{nice} invariant types together with the domination relation. For example, the space of invariant types in o-minimal structures and henselian valued fields~\cite{rosario-omin,rosario-hils} has been studied in precisely this way. Therefore, the question of whether a notion of domination exists for Keisler measures naturally arises. Moreover, in the context of valued fields, a useful tool in studying the geometry of definable sets is the notion of \textit{stable domination}. An extensive study on this subject has been made in~\cite{HHM-stable-domination} which contributed to the understanding of Berkovich analytification of  varieties~\cite{HruLoe-tame-top}. It is conceivable that understanding domination of measures in the valued field context may yield an understanding of certain associated geometric spaces. More specifically, as suggested by the previous works on types~\cite{HruLoe-tame-top,beautiful-pairs}, it would be nice to be able to internalize the space of generically stable measures. To put a more concrete question:
\begin{question}
Is every generically stable measure in ACVF dominated by a collection of stably dominated types?
\end{question}

    A positive answer would suggest that generically stable measures can indeed be approached as small projective limits of definable sets with some small auxilary data, which would serve as the first step towards an analysis of generically stable measures following~\cite{HruLoe-tame-top}.

The notion of extension domination is also conceptually related to compact domination (see \cite{NIP-III,simon2017vc} and \cite[Chapter 8]{simon2015guide}). Compact domination expresses the idea that a Keisler measure can be \textit{controlled} by an auxiliary measure over a compact space. Extension domination, on the other hand, attempts to capture the idea of when one Keisler measure controls another. With compact domination, a measure $h$ on a compact space $K$ lifts uniquely to an invariant Keisler measure. With extension domination, if $\mu_{x}$ extension dominates $\nu_{y}$, then there exists a unique Keisler measure $\omega_{xy}$ whose projection to $x$ is $\mu_x$ and projection to $y$ is $\nu_y$.

In this paper, we introduce \textit{extension domination} for Keisler measures as a natural generalization of domination for types. We show that this notion extends domination for types (Proposition \ref{dom:exdom}) and forms a preorder on the space of global Keisler measures (Theorem \ref{Thm:pre}). Additionally, any Keisler measure dominates (definable) pushforwards of itself (Proposition \ref{pushforward}). Extension domination also admits a natural approximation property: If $\mu$ extension dominates $\nu$, then there is a collection of formulas which forces this to be true via \textit{squeezing} (Proposition \ref{approx:1} and Theorem \ref{approx:2}). Moreover, several properties of measures such as invariance, smoothness, and finite satisfiability are preserved under extension domination (Section 4). Akin to the classical setting where realized types form the bottom tier of the domination relation, the collection of smooth measures forms the bottom tier of the extension domination relation (Proposition \ref{smooth:pres} and Proposition \ref{smooth:2}). 

If domination is witnessed by a separated amalgam (such as the restriction of a Morley product), then one can show that extension domination is closed under localizations (Proposition \ref{local:ext}). The role of the separated amalgam allows us more control over our computation and leads to the notion of \textit{amalgam deciphering}. A measure $\mu$ amalgam deciphers a measure $\nu$ if $\mu$ dominates $\nu$ over the class of separated amalgams (instead of the class of all possible extensions). The general behavior of this notion remains somewhat mysterious, yet it interacts well with localizations (Proposition \ref{amal:local}) and convex combinations (Proposition \ref{prop:conv}).


The paper is structured as follows: In Section~\ref{sec:prelim}, we recall some preliminaries and basic facts about domination, finitely additive measures, and Keisler measures. In Section~\ref{sec:extension}, we introduce and study extension domination. In Section~\ref{sec:preserve}, we prove some preservation theorems.  In Section~\ref{sec:amalgam}, we study separated amalgams and amalgam deciphering. In Section 6, we provide some explicit examples of extension domination. 

\section{Preliminaries}\label{sec:prelim}

For the most part, our notation is standard. If $r,s$ are real numbers and $\epsilon > 0$, we write $r \approx_{\epsilon} s$ to mean $|r - s| < \epsilon$. We let $T$ denote a complete first order theory in a language $\mathcal{L}$ and $\mathcal{U}$ denote a monster model of $T$. We use $x$, $y$, and $z$ to denote finite tuples of variables. If $A \subseteq \mathcal{U}$, we let $\mathcal{L}_{x}(A)$ denote the Boolean algebra of formulas with free variable(s) $x$ and parameters from $A$ (modulo logical equivalence by $T$). We also identify $\mathcal{L}_{x}(A)$ with the Boolean algebra of $A$-definable subsets of $\mathcal{U}^{x}$ without comment. We let $S_{x}(A)$ be the collection of types over $A$ in free variable(s) $x$. 

If $\varphi(x) \in \mathcal{L}_{x}(A)$, we let $[\varphi(x)] := \{p \in S_{x}(A): \varphi(x) \in p\}$. If $\varphi(x)$ is a formula, $\varphi^{1}(x) = \varphi(x)$ and $\varphi^{0}(x) = \neg \varphi(x)$. 

\subsection{Domination for types}
We first recall the definition of domination for global types. We refer the reader to~\cite{rosario-inv-dom} for additional details.

\begin{definition}\label{def:type-dom}
Let $p\in S_x(\mathcal{U})$, $q\in S_y(\mathcal{U})$ and $A$ be a small subset of $\mathcal{U}$. We say that \emph{$p$ dominates $q$ (over $A$)} if there is some $r\in S_{xy}(A)$ such that  $p|_{A} \subseteq r$ and $p\cup r\vdash q$. We write $p\geq_{D,A}q$ to mean ``$p$ dominates $q$ over $A$" and $p \geq_{D} q$ to mean ``there exists some small set $A$ such that $p \geq_{D,A}q$". Finally, if $r$ is as above we say that the type $r$ \emph{witnesses $p\geq_{D,A}q$} or simply \emph{witnesses domination}.
\end{definition} 

The next fact follows directly from compactness.

\begin{fact}\label{type:approx} Let $p \in S_{x}(\mathcal{U})$, $q \in S_{y}(\mathcal{U})$, $A$ be a small subset of $\mathcal{U}$, and $r \in S_{xy}(A)$ such that  $r$ witnesses $p \geq_{D,A} q$. Then for every $\psi(y) \in q$, there exists $\varphi(x) \in p$ and $\theta(x,y) \in r$ such that $\varphi(x) \wedge \theta(x,y) \subseteq  \psi(y) \wedge x = x$.  
\end{fact}

\begin{remark}\label{rmk:type-dom} 
We remark that in the definition of domination in~\cite{rosario-inv-dom}, there is an additional assumption that $r\in S_{pq}(A):=\{r' \in S_{xy}(A):r'\supseteq p|_A\cup q|_A\}.$ The assumption that $r\supseteq q|_A$ is not strictly necessary since this property is a consequence of domination as defined above.
\end{remark}

\begin{remark}
By compactness, Definition~\ref{def:type-dom} is equivalent to the following: $p\geq_{D,A} q$ if and only if there is some $r\in S_{xy}(A)$, such that $r\supseteq p|_A$ and for any $w\in S_{xy}(\mathcal{U})$ such that $w$ extends both $p$ and $r$, $w$ extends $q$. This rephrasing of the definition motivates our definition for extension domination (see Definition~\ref{def:ext-dom}). 
\end{remark}

\subsection{Finitely additive measures}
We now recall some preliminaries about finitely additive measures. A general reference for these facts can be found in \cite{rao1983theory}.

\begin{fact}[{\cite[Theorem 3.6.1]{rao1983theory}}]\label{fact:comm-ext} Let $X$ be a set, and $\mathcal{A}$ and $\mathcal{B}$ be Boolean algebras on $X$ which contain $\{X\}$. Let  $\mu_1$ and $\mu_2$ be two bounded (positive) finitely additive measures on $\mathcal{A}$ and $\mathcal{B}$ respectively. Let $\mathcal{C}$ be a Boolean algebra on $X$ containing both $\mathcal{A}$ and $\mathcal{B}$. Then there exists a bounded (positive) finitely additive measure $\mu_{3}$ on $\mathcal{C}$ which is a common extension of $\mu_1$ and $\mu_2$ if and only if for every $A \in \mathcal{A}$ and $B \in \mathcal{B}$, if $A \subseteq B$, then $\mu_1(A) \leq \mu_2(B)$ and if $B \subseteq A$, then $\mu_2(B) \leq \mu_1(A)$. 
\end{fact}

\begin{corollary}\label{Cor:exist} Let $X$ be a set, and $\mathcal{A}$ and $\mathcal{B}$ be Boolean algebras of sets on $X$ which contain $\{X\}$. Let $\mu_1$ and $\mu_2$ be two finitely additive probability measures on $\mathcal{A}$ and $\mathcal{B}$ respectively. Let $\mathcal{C}$ be a Boolean algebra on $X$ containing both $\mathcal{A}$ and $\mathcal{B}$. Then there exists a finitely additive probability measure $\mu_{3}$ on $\mathcal{C}$ which is a common extension of $\mu_1$ and $\mu_2$ if for every $A \in \mathcal{A}$ and $B \in \mathcal{B}$, if $A \subseteq B$, then $\mu_1(A) \leq \mu_2(B)$. 
\end{corollary} 

\begin{proof} By Fact \ref{fact:comm-ext}, it suffices to show that for every $A \in \mathcal{A}$ and $B \in \mathcal{B}$ such that $B \subseteq A$, $\mu_2(B) \leq \mu_1(A)$. Suppose $B \subseteq A$. Then $A^{c} \subseteq B^{c}$ and so $\mu_1(A^{c}) \leq \mu_2(B^{c})$. Hence $1 - \mu_1(A) \leq 1 - \mu_2(B)$ and therefore $\mu_2(B) \leq \mu_1(A)$. 

By Fact~\ref{fact:comm-ext}, there is a bounded (positive) finitely additive measure $\mu_3$ on $\mathcal{C}$ which extends both $\mu_1$ and $\mu_2$. Notice  $\mu_3$ is a probability measure since $\mu_3(X) = \mu_1(X) = 1$. 
\end{proof} 

\begin{fact}[{\cite[Theorem 3.3.3]{rao1983theory}}]\label{fact:charges-r}  Let $X$ be a set, $\mathcal{A}$ and $\mathcal{B}$ be Boolean algebras of sets on $X$ which contain $\{X\}$, $\mathcal{A} \subseteq \mathcal{B}$,  and  $\mu$ be a finitely additive probability measure on $\mathcal{A}$. Fix $B \in \mathcal{B}$ and let
\begin{equation*} r_1 = \sup \{\mu(A): A \in \mathcal{A}, A \subseteq B\} \text{ and } r_2 = \inf\{\mu(A): A \in \mathcal{A}, B \subseteq A\}. 
\end{equation*} 
Then $r_1 \leq r_2$ and for any $d$ such that $r_1 \leq d \leq r_2$, there exists a finitely additive probability measure $\mu_d$ on $\mathcal{B}$ such that $\mu_{d}|_{\mathcal{A}} = \mu$ and $\mu_{d}(B) = d$. 
\end{fact}

The following fact is used throughout the paper without comment. 

\begin{fact} Let $A \subseteq \mathcal{U}$. The Boolean algebra $\mathcal{L}_{x}(A)$ injectively embeds into $\mathcal{L}_{xy}(A)$ via \begin{equation*}
    \alpha(\varphi(x)) := \varphi(x) \wedge y = y.
\end{equation*}
\end{fact} 

\subsection{Keisler measures}

Lastly we recall some basic definitions and standard observations involving Keisler measures. 

\begin{definition} Let $A \subseteq \mathcal{U}$ and $x$ a tuple of variables. Then a \textit{Keisler measure} (over $A$, in variable(s) $x$) is a finitely additive probability measure  on $\mathcal{L}_{x}(A)$. More explicitly, a Keisler measure $\mu$ is a map from $\mathcal{L}_{x}(A)$ to $[0,1]$ such that
\begin{enumerate}
    \item $\mu( x = x) = 1$.
    \item For any $\varphi(x) \in \mathcal{L}_{x}(A)$, $\mu(\neg \varphi(x)) = 1 - \mu(\varphi(x))$. 
    \item For any $\varphi(x), \psi(x) \in \mathcal{L}_{x}(A)$, 
    \begin{equation*}
        \mu(\varphi(x) \vee \psi(x)) = \mu(\varphi(x)) + \mu(\psi(x)) - \mu(\varphi(x) \wedge \psi(x)). 
    \end{equation*}
\end{enumerate}We denote the space of Keisler measures on $\mathcal{L}_{x}(A)$ as $\mathfrak{M}_{x}(A)$. 
\end{definition}

\begin{definition}\label{support} Let $\mu \in \mathfrak{M}_{x}(A)$ and $p \in S_{x}(A)$. We say that $p$ \textit{ is in the support of $\mu$} if for every $\varphi(x) \in p$, $\mu(\varphi(x)) > 0$. We let $\supp(\mu)$ be the collection of types in $S_{x}(A)$ which are in the support of $\mu$. We remark that for any $\mu \in \mathfrak{M}_{x}(A)$, $\supp(\mu) \neq \emptyset$ and $\supp(\mu)$ is a closed subset of $S_{x}(A)$. 
\end{definition}

We now discuss the topological/geometric properties of the space of Keisler measures. The following observations are folklore and follow from basic results in classical analysis. We refer the reader to \cite[Chapter 7]{simon2015guide} for more details. We first recall some definitions from functional analysis.

\begin{definition} Let $(V,||\cdot||)$ be a real Banach space and $V^{*}$ be the dual of $V$. More explicitly, $V^{*}$ is the collection of (norm)-continuous linear maps from $V$ to $\mathbb{R}$. Then the weak$^*$ topology on $V^{*}$ is the coarsest topology on $V^{*}$ such that for any $x \in V$, the map $f \mapsto f(x)$ is continuous.  
\end{definition} 

\begin{fact}\label{Basic:Keisler} Let $A \subseteq \mathcal{U}$. 
\begin{enumerate}
    \item $\mathfrak{M}_{x}(A)$ forms a compact Hausdorff space with the following topology: Let $\varphi_1(x),...,\varphi_n(x)$ be formulas in $\mathcal{L}_{x}(A)$ and $r_1,...,r_n,s_1,...,s_n$ be real numbers. Then a basic open set $O$ of $\mathfrak{M}_{x}(A)$ is of the form
    \begin{equation*}
        O = \bigcap_{i=1}^{n} \{\nu \in \mathfrak{M}_{x}(A): r_i < \nu(\varphi_i(x)) < s_i\}. 
    \end{equation*}
    \item For any $\varphi(x) \in \mathcal{L}_{x}(A)$, the map $E_{\varphi}:\mathfrak{M}_{x}(A) \to [0,1]$ via $E_{\varphi}(\mu) = \mu(\varphi(x))$ is continuous.
    \item Every Keisler measure $\mu$ in $\mathfrak{M}_{x}(A)$ is in unique correspondence with a regular Borel probability measure on $S_{x}(A)$. We denote the space of regular probability measures on $S_{x}(A)$ as $\mathcal{M}_{x}(A)$. 
    \item If we let $V = (C(S_{x}(A)),||\cdot||_{\infty})$ be the Banach space of continuous functions from $S_{x}(A)$ to $\mathbb{R}$ with the supremum norm, then $\mathcal{M}_{x}(A)$ is naturally a subset of $V^{*}$ with the weak$^*$ topology. By (3), the space $\mathfrak{M}_{x}(A)$ can be identified with a subset of $V^{*}$. The induced topology on $\mathfrak{M}_{x}(A)$ from $V^{*}$ is the same as the topology described in (1). Moreover,  $\mathfrak{M}_{x}(A)$ is a convex subset of $V^{*}$.
    \item There is a topological embedding $\delta: S_{x}(A) \to \mathfrak{M}_{x}(A)$ via $p \to \delta_{p}$ where $\delta_{p}$ is a the Dirac measure concentrating on $p$. In other words, for any $\varphi(x) \in \mathcal{L}_{x}(A)$,
    \begin{equation*}
        \delta(p)(\varphi(x)) = \delta_{p}(\varphi(x))=\begin{cases}
\begin{array}{cc}
1 & \text{if }\varphi(x)\in p,\\
0 & \text{otherwise.}
\end{array}\end{cases}
    \end{equation*}
\end{enumerate}
\end{fact} 

\begin{remark} As written, Fact \ref{Basic:Keisler} is first explained in the model theory context in Section 4 of \cite{hrushovski2011nip}. However, this theorem can be found in the Stone space literature going back much earlier \cite[Theorem 2.2.2]{chan1977measure}. 
\end{remark}

\begin{definition} Let $V_1$ and $V_2$ be locally convex topological vector spaces. Let $E_1 \subseteq V_1$ and $E_2 \subseteq V_2$. We say that $\alpha:E_1 \to E_2$ is an affine homeomorphism if
    \begin{enumerate}
        \item $\alpha$ is a homeomorphism from $E_1$ to $E_2$ with the induced topologies. 
        \item $\alpha$ is affine, i.e. for $a,b \in E_1$ and $r \in \mathbb{R}$,  $\alpha(ra + b) = r\alpha(a) + \alpha(b)$. 
    \end{enumerate}
\end{definition}

\begin{fact} Let $B \subseteq A \subseteq \mathcal{U}$. \begin{enumerate}
    \item The natural restriction map $r_{A,B}: \mathfrak{M}_{x}(A) \to \mathfrak{M}_{x}(B)$ is continuous. If $\mu \in \mathfrak{M}_{x}(A)$, we usually write $r_{A,B}(\mu)$ simply as $\mu|_{B}$. 
    \item The natural projection map $\pi_{x}: \mathfrak{M}_{xy}(A) \to \mathfrak{M}_{x}(A)$, defined via $\pi_x(\mu)(\varphi(x))=\mu(\varphi(x)\wedge y=y)$ is continuous. 
    \item Restriction maps and projection maps commute: If $\omega \in \mathfrak{M}_{xy}(A)$, then $\pi_{x}(\omega)|_{B} = \pi_{x}(\omega|_{B})$. 
\end{enumerate}
\end{fact}

The following facts are basic and left as an exercise. These facts are often used without mention throughout the paper. 

\begin{fact}\label{Basic:Fact} Let $\omega \in \mathfrak{M}_{xy}(A)$, $\varphi(x) \in \mathcal{L}_{x}(A)$, and $\theta(x,y),\rho(x,y) \in \mathcal{L}_{xy}(A)$. Then 
\begin{enumerate} 
\item If $\pi_{x}(\omega)(\varphi(x)) = 1$, then $\omega(\varphi(x) \wedge \theta(x,y)) = \omega(\theta(x,y))$. 
\item $\omega(\theta(x,y)) \leq \pi_{x}(\omega)(\exists y(\theta(x,y))$.
\item If $\omega(\theta(x,y)) = 1$, then $\omega(\rho(x,y) \wedge \theta(x,y)) = \omega(\rho(x,y))$. 
\end{enumerate} 
\end{fact} 

Finally, we recall the definition of a separated amalgam \cite[Definition 4.2]{NIP-III}. 
\begin{definition}\label{def:amal}
Let $\lambda \in \mathfrak{M}_{xy}(A)$. We say that $\lambda$ is a \emph{separated amalgam} if for any pair of formulas $\varphi(x) \in \mathcal{L}_{x}(A)$ and $\psi(y) \in \mathcal{L}_{y}(A)$, we have that 
\begin{equation*}
    \lambda(\varphi(x) \wedge \psi(y)) = \pi_{x}(\lambda)(\varphi(x)) \cdot \pi_{y}(\lambda)(\psi(y)). 
\end{equation*} We let $\mathfrak{M}^{\mathrm{Am}}_{xy}(A) := \{\lambda \in \mathfrak{M}_{xy}(A): \lambda$ is a separated amalgam$\}$. 
\end{definition}

\begin{lemma}\label{Amal:eq} The following are equivalent: 
\begin{enumerate}
    \item $\lambda \in \mathfrak{M}^{\mathrm{Am}}_{xy}(A)$. 
    \item There exists measures $\mu \in \mathfrak{M}_{x}(A)$ and $\nu \in \mathfrak{M}_{y}(A)$ such that for any $(\varphi(x),\psi(y)) \in \mathcal{L}_{x}(A) \times \mathcal{L}_{y}(A)$, 
    $\mu(\varphi(x)) \cdot \nu(\psi(y)) = \lambda(\varphi(x) \wedge \psi(y))$. 
\end{enumerate}
\end{lemma}

\begin{proof} $(1) \implies (2)$ is trivial. Suppose $(2)$ holds. Then 
\begin{equation*}
    \mu(\varphi(x)) = \mu(\varphi(x)) \cdot \nu(y=y) = \lambda(\varphi(x) \wedge y = y) = \pi_{x}(\lambda)(\varphi(x)). 
\end{equation*}
Similarly, $\nu(\psi(y)) = \pi_{y}(\lambda)(\psi(y))$. Hence $\lambda$ is a separated amalgam. 
\end{proof}

\begin{remark}\label{samalgam:exist} Let $A \subseteq \mathcal{U}$. For any pair of measures $\mu \in \mathfrak{M}_{x}(A)$ and $\nu \in \mathfrak{M}_{y}(A)$, there exists a measure $\omega \in \mathfrak{M}^{\am}_{xy}(A)$ such that $\pi_{x}(\omega) = \mu$ and $\pi_{y}(\omega) = \nu$. Indeed, one can build a finitely additive product measure $\mu \times \nu$ on $\mathcal{L}_{x}(A) \times \mathcal{L}_{y}(A)$. The map $\alpha: \mathcal{L}_{x}(A) \times \mathcal{L}_{y}(A) \to \mathcal{L}_{xy}(A)$ given by $(\varphi(x),\psi(y)) \to \varphi(x) \wedge \psi(y)$ is an injective homomorphism of Boolean algebras. Hence we can consider the measure $\alpha(\mu \times \nu)$ as a finitely additive probability measure on the 
subalgebra of $\mathcal{L}_{xy}(A)$ generated by $\{\varphi(x) \wedge \psi(y): \varphi(x) \in \mathcal{L}_{x}(A), \psi(y) \in \mathcal{L}_{y}(A)\}$. Any $\omega$ extending $\alpha(\mu \times \nu)$ to the Boolean algebra $\mathcal{L}_{xy}(A)$ is a separated amalgam. Notice that $\omega$ exists by an application of Corollary \ref{Cor:exist} with $\mathcal{A} = \mathcal{B} = \alpha(\mathcal{L}_{x}(A) \times \mathcal{L}_{y}(A))$, $\mu_1 = \mu_2 = \alpha(\mu \times \nu)$, and $\mathcal{C} = \mathcal{L}_{xy}(A)$. 
\end{remark} 

The following is stated in passing after Definition 2.4 in \cite{NIP-III}. We provide a proof. 

\begin{lemma}\label{type:amal} Let $A \subseteq \mathcal{U}$ and $\lambda \in \mathfrak{M}_{xy}(A)$.  Suppose either $\pi_{x}(\lambda) = \delta_{p}$ for some $p \in S_{x}(A)$ or $\pi_{y}(\lambda) = \delta_{q}$ for some $q \in S_{y}(A)$. Then $\lambda \in \mathfrak{M}_{xy}^{\am}(A)$.  
\end{lemma} 

\begin{proof} By symmetry, we assume that $\pi_{x}(\lambda) = \delta_{p}$. Fix $(\varphi(x),\psi(y)) \in \mathcal{L}_{x}(A) \times \mathcal{L}_{y}(A)$. First suppose that $\varphi(x) \in p$. Then
\begin{equation*} 
\pi_{x}(\lambda)(\varphi(x)) \cdot \pi_{y}(\lambda)(\psi(y)) = \delta_{p}(\varphi(x)) \cdot \pi_{y}(\lambda)(\psi(y)) = \pi_{y}(\lambda)(\psi(y)) 
\end{equation*} 
\begin{equation*}
    = \lambda(\psi(y) \wedge x= x) \overset{(*)}{=} \lambda(\varphi(x) \wedge \psi(y)). 
\end{equation*}
Where equality $(*)$ holds from (1) of Lemma \ref{Basic:Fact} and $\pi_{x}(\lambda)(\varphi(x)) = \delta_{p}(\varphi(x)) = 1$. 

Now suppose that $\neg \varphi(x) \in p$. Clearly $ \delta_{p}(\varphi(x)) \cdot \pi_{y}(\lambda)(\psi(y)) = 0$. Also note,  
\begin{equation*} 
 0 = \delta_{p}(\varphi(x)) = \pi_{x}(\lambda)(\varphi(x)) = \lambda(\varphi(x) \wedge y = y) \geq \lambda(\varphi(x) \wedge \psi(y))
\end{equation*} 
\begin{equation*}
    \implies \lambda(\varphi(x) \wedge \psi(y)) = 0. 
\end{equation*}
Hence $\lambda \in \mathfrak{M}_{xy}^{\am}(A)$. 
\end{proof} 

\section{Extension Domination}\label{sec:extension}
In this section, we introduce the notion of \textit{extension domination} which we denote as `$\geq_{\mathbb{E}}$'. We then prove that extension domination \textit{behaves} like a notion of domination. More explicitly, we show that
\begin{enumerate}
    \item Extension domination extends the definition of domination for global types. 
    \item Any Keisler measure dominates its (definable) push-forwards. 
    \item Extension domination forms a preorder on global Keisler measures. 
\end{enumerate} 
After proving these properties, we prove a general approximation theorem. This theorem essentially states that if a measure $\mu$ extension dominates a measure $\nu$, then there is a collection of formulas which forces this to be true via \textit{squeezing}. We also derive a uniform version of this theorem via compactness. We end this section by showing that under a moderately stronger assumption, extension domination interacts nicely with  \textit{localizations}, a phenomenon unique to measures. This occurs when domination is witnessed by a separated amalgam. 

\subsection{Basic definitions and observations} Throughout this section, $\mathcal{U}$ is a fixed monster model and $A$ is a small subset of $\mathcal{U}$.  In order to define extension domination, we must first define the notion of an \textit{extension space}. 

\begin{definition} Let $\mu \in \mathfrak{M}_{x}(\mathcal{U})$ and $\lambda \in \mathfrak{M}_{xy}(A)$ such that $\mu|_{A} = \pi_{x}(\lambda)$. We define the corresponding \textit{extension space}, denoted $\E(\lambda,\mu)$, as follows:
\begin{equation*}\E(\lambda,\mu) := \{\omega \in \mathfrak{M}_{xy}(\mathcal{U}): \omega|_{A} = \lambda, \pi_{x}(\omega) = \mu\}.
\end{equation*} 
\end{definition} 

Notice that in the definition of an extension space, we require a consistency condition, i.e. $\mu|_{A} = \pi_{x}(\lambda)$. We will now argue that this condition is sufficient to ensure that $\E(\lambda,\mu)$ is non-empty. We also observe that $\E(\lambda,\mu)$ is a convex compact space. 

\begin{lemma}\label{ext:exist} Let $\mu \in \mathfrak{M}_{x}(\mathcal{U})$ and $\lambda \in \mathfrak{M}_{xy}(A)$ such that $\mu|_{A} = \pi_{x}(\lambda)$. Then
\begin{enumerate} 
\item $\E(\lambda,\mu) \neq \emptyset$. 
\item $\E(\lambda,\mu)$ is a compact Hausdorff space under the topology induced by $\mathfrak{M}_{xy}(\mathcal{U})$. 
\item $\E(\lambda, \mu)$ is convex. 
\end{enumerate} 
\end{lemma}
\begin{proof} We prove the statements. 
\begin{enumerate}
    \item This follows from an application of Corollary \ref{Cor:exist}. Identify the Boolean algebras $\mathcal{L}_{x}(\mathcal{U})$ and $\mathcal{L}_{xy}(A)$ with their images under their natural embeddings into $\mathcal{L}_{xy}(\mathcal{U})$ as well as $\mu$ and $\lambda$ with the appropriate finitely additive measures on the corresponding images. Let $\theta(x,y) \in \mathcal{L}_{xy}(A)$, $\varphi(x) \in \mathcal{L}_{x}(\mathcal{U})$, and suppose that $\theta(x,y) \subseteq   \varphi(x) \wedge y =y$ as definable sets. Consider the formula $\gamma(x,y) := \exists y \theta(x,y) \wedge y = y$. Notice that $\theta(x,y) \subseteq \gamma(x,y) \subseteq \varphi(x) \wedge y =y$. Now
\begin{equation*}
    \lambda(\theta(x,y)) \leq \lambda(\gamma(x,y)) = \pi_{x}(\lambda)(\exists y\theta(x,y)) 
\end{equation*}
\begin{equation*}
    = \mu(\exists y \theta(x,y)) \leq \mu(\varphi(x)) = \mu(\varphi(x) \wedge y = y).
\end{equation*}
By Corollary \ref{Cor:exist}, there exists $\omega \in \mathfrak{M}_{xy}(\mathcal{U})$ which is a common extension of the $\lambda$ and $\mu$. By construction, $\omega \in \E(\lambda,\mu)$. 
\item Any subspace of a Hausdorff space is Hausdorff. Let $r_{A}:\mathfrak{M}_{xy}(\mathcal{U}) \to \mathfrak{M}_{xy}(A)$ be the obvious restriction map. Then $\E(\lambda,\mu) =\pi_{x}^{-1}(\{\mu\}) \cap (r_{A})^{-1}(\{\lambda\})$. Since both $r_{A}$ and $\pi_{x}$ are continuous, $\E(\lambda,\mu)$ is the intersection of two compact sets (and thus compact).
\item Suppose that $\omega_1, \omega_2 \in \E(\lambda,\mu)$ and let $r,s \in [0,1]$ such that $r + s =1$. 
Then 
\begin{equation*} 
\pi_{x}(r\omega_1 + s\omega_2) = r\pi_{x}(\omega_1) + s\pi_{x}(\omega_1) = r\mu + s\mu = \mu,
\end{equation*} 
and, 
\begin{equation*} 
(r\omega_1 + s\omega_1)|_{A} = r\omega_1|_{A} + s\omega_1|_{A} = r\lambda + s\lambda = \lambda.
\end{equation*} 
Hence $r\omega_1 + s\omega_2 \in \E(\lambda,\mu)$.  \qedhere
\end{enumerate} 
\end{proof}

We now define extension domination.

\begin{definition}\label{def:ext-dom} Let $\mu\in \mathfrak{M}_x(\mathcal{U})$, $\nu\in \mathfrak{M}_y(\mathcal{U})$, and $A$ be a small subset of  $\mathcal{U}$. We say that $\mu$ \textit{extension dominates $\nu$ (over $A$)} (denoted by $\mu \geq_{\mathbb{E},A} \nu$) if there exists $\lambda \in \mathfrak{M}_{xy}(A)$ such that 
\begin{enumerate}
\item $\pi_{x}(\lambda) = \mu|_{A}$. 
\item For any $\omega$ in $\E(\lambda, \mu)$, $\pi_{y}(\omega) = \nu$. 
\end{enumerate}
If the above conditions are satisfied, then $\lambda$ is said to \textit{witness} $\mu \geq_{\mathbb{E},A} \nu$ or \textit{witness domination}. We write $\mu \geq_{\mathbb{E}} \nu$ if there exists some small set $A$ such that $\mu \geq_{\mathbb{E},A} \nu$. 
\end{definition} 

\begin{convention} Let $\mathfrak{M}(\mathcal{U})$ be the space of global Keisler measures in finitely many variables. Suppose that $\mu$ and $\nu$ are in $\mathfrak{M}(\mathcal{U})$. We often consider $\geq_{\mathbb{E},A}$ and $\geq_{\mathbb{E}}$ as  binary relations on this space by saying that $\mu$ dominates $\nu$ (over $A$) if $\mu(x)$ dominates $\nu(y)$ (over $A$) where $x$ and $y$ are distinct tuples of variables.
\end{convention}

\begin{remark} We remark that extension domination is \textit{non-vacuous}. More explicitly, if $\mu \geq_{\mathbb{E},A} \nu$ and $\lambda$ witnesses domination, then $\E(\lambda,\mu)$ is non-empty. This is precisely the conclusion of $(1)$ of Lemma \ref{ext:exist}. This will become an issue when we discuss \textit{amalgam deciphering} in Section 5. 
\end{remark} 

We now show that extension domination extends the classical definition of domination for types.

\begin{proposition}\label{dom:exdom} Let $p\in S_x(\mathcal{U})$, $q\in S_y(\mathcal{U})$, and $A$ be a small subset of $\mathcal{U}$. If $p \geq_{D,A} q$, then $\delta_{p} \geq_{\mathbb{E},A} \delta_{q}$. 
\end{proposition} 

\begin{proof} Suppose that $p\geq_{D,A}q$. Then there exists some complete type $r \in S_{xy}(A)$ which witnesses $p \geq_{D,A} q$. Consider the measure $\lambda : = \delta_{r}$. It is clear from construction that $\lambda \in \mathfrak{M}_{xy}(A)$ and $\pi_{x}(\lambda) = \delta_{p}|_{A}$. Fix $\omega \in \E(\lambda,\delta_{p})$. It suffices to show that $\pi_{y}(\omega) = \delta_{q}$. Fix a formula $\psi(y) \in \mathcal{L}_{y}(\mathcal{U})$ such that $\delta_{q}(\psi(y)) = 1$. By definition, $\psi(y) \in q$ and since $p \geq_{D,A} q$, there exists $\varphi(x) \in p$ and $\theta(x,y) \in r$ such that $\varphi(x) \wedge \theta(x,y) \subseteq \psi(y) \wedge x = x$ (see Fact \ref{type:approx}). Hence 
\begin{equation*}
    \pi_{y}(\omega)(\psi(y))  = \omega(\psi(y) \wedge x=x) \geq \omega(\varphi(x) \wedge \theta(x,y)) 
\end{equation*}
\begin{equation*}
    \overset{(*)}{=} \omega(\varphi(x) \wedge y =y) = \pi_{x}(\omega)(\varphi(x)) = \delta_{p}(\varphi(x)) =  1.
\end{equation*}
\begin{equation*}
    \implies \pi_{y}(\omega)=  1.
\end{equation*}
\vspace{.001in}
Equality $(*)$ follows from the fact that $1 = \lambda(\theta(x,y)) = \omega(\theta(x,y))$. The case where $\delta_q(\psi(y))=0$ is handled via negation.
\end{proof}
\begin{question} Suppose that $\delta_{p} \geq_{\mathbb{E}} \delta_{q}$. Does it follow that $p \geq_{D} q$?
\end{question}

Our next lemma allows us to extend the space of parameters.

\begin{lemma}\label{lem:ext-para} Let $A$ and $B$ be a small subsets of $\mathcal{U}$ such that $A \subseteq B$. If $\mu \geq_{\mathbb{E},A} \nu$, then $\mu \geq_{\mathbb{E},B} \nu$.
\end{lemma} 
\begin{proof} Since $\mu \geq_{\mathbb{E},A} \nu$, there exists some $\lambda \in \mathfrak{M}_{xy}(A)$ which witnesses domination. By Lemma \ref{ext:exist}, there exists some $\omega \in \E(\lambda,\mu)$. Let $\lambda' := \omega|_{B}$.  We claim that $\lambda'$ witnesses $\mu \geq_{\mathbb{E},B} \nu$.
\begin{enumerate} 
\item Notice that 
\begin{equation*} 
\pi_{x}(\lambda') = \pi_{x}(\omega|_{B}) = \pi_{x}(\omega)|_{B} = \mu|_{B}. 
\end{equation*} 
\item Let $\eta \in \E(\lambda',\mu)$. It suffices to show that $\pi_{y}(\eta) = \nu$. Since $\eta \in \E(\lambda',\mu)$, it follows that $\pi_{x}(\eta) = \mu$ and $\eta|_{A} = (\eta|_{B})|_{A} = \lambda'|_A =  \lambda$.  Hence $\eta \in \E(\lambda,\mu)$ and since $\mu$ dominates $\nu$ over $A$ (via $\lambda$), we conclude $\pi_{y}(\eta) = \nu$. \qedhere
\end{enumerate} 
\end{proof}

We now argue that measures dominate push-forwards of themselves (under definable maps). We recall the definition of a push-forward. 

\begin{definition}\label{def:push-for} Let $\mu \in \mathfrak{M}_{x}(\mathcal{U})$ and $f: \mathcal{U}^{x} \to \mathcal{U}^{y}$ be a definable function (possibly defined over parameters). We define the push-forward measure $f(\mu) \in \mathfrak{M}_{y}(\mathcal{U})$ where for any $\psi(y) \in \mathcal{L}_{y}(\mathcal{U})$,  
\begin{equation*} f(\mu)(\psi(y)) = \mu(\psi(f(x)). 
\end{equation*} 
\end{definition} 

\begin{proposition}\label{pushforward} Let $\mu \in \mathfrak{M}_{x}(\mathcal{U})$ and $f: \mathcal{U}^{x} \to \mathcal{U}^{y}$ be an $A$-definable function. Then $\mu \geq_{\mathbb{E},A} f(\mu)$.
\end{proposition} 

\begin{proof} Consider the function $g: \mathcal{U}^{x} \to \mathcal{U}^{x} \times \mathcal{U}^{y}$ via $g(a) = (a,f(a))$. The map $g$ is an $A$-definable function and so we can consider the push-forward $g(\mu)$. Define $\lambda: = g(\mu)|_{A}$. Notice that 
\begin{enumerate} 
\item $\pi_{x}(\lambda) = \mu|_{A}$ and $\pi_{y}(\lambda) = f(\mu)|_{A}$. 
\item $\lambda(f(x) = y) = 1$. 
\end{enumerate} 
Fix $\omega \in \E(\lambda,\mu)$. Consider the following computation:
\begin{equation*} \pi_{y}(\omega)(\psi(y)) = \omega(\psi(y) \wedge x = x) \overset{(*)}{=} \omega(\psi(y) \wedge f(x) = y) = \omega(\psi(f(x)) \wedge f(x) = y)
\end{equation*} 
\begin{equation*} \overset{(*)}{=} \omega(\psi(f(x)) \wedge y = y) = \pi_{x}(\omega)(\psi(f(x)) = \mu(\psi(f(x))) = f(\mu)(\psi(y)),
\end{equation*} 
where the equations with $(*)$ hold by Fact \ref{Basic:Fact} and the observation that $\omega(f(x) =y) = \lambda(f(x) = y) = 1$. 
\end{proof}

\subsection{Extension domination forms a preorder}

We show that $\geq_{\mathbb{E}}$ is a preorder on the collection of global measures, $\mathfrak{M}(\mathcal{U})$. More explicitly, we argue that the relation $\geq_{\mathbb{E}}$ is both transitive and reflexive. We begin with a measure existence lemma.

\begin{lemma}\label{ext:lemmas} Let $A \subseteq \mathcal{U}$. 
\begin{enumerate}
\item Let $\lambda_1 \in \mathfrak{M}_{xy}(A)$ and $\lambda_2 \in \mathfrak{M}_{yz}(A)$ such that $\pi_{y}(\lambda_1) = \pi_{y}(\lambda_2)$. Then there exists $\lambda_3 \in \mathfrak{M}_{xyz}(A)$ such that $\pi_{xy}(\lambda_3) = \lambda_1$ and $\pi_{yz}(\lambda_3) = \lambda_2$. 
\item Suppose that $\lambda \in \mathfrak{M}_{xyz}(A)$, $\omega \in \mathfrak{M}_{xz}(\mathcal{U})$, and $\omega|_{A} = \pi_{xz}(\lambda)$. Then there exists some measure $\omega' \in \mathfrak{M}_{xyz}(\mathcal{U})$ such that $\pi_{xy}(\omega') = \omega$ and $\omega'|_{A} = \lambda$. 
\end{enumerate} 
\end{lemma} 

\begin{proof} The proofs of the claims are similar to one another and also to Lemma \ref{ext:exist}. For clarity, we provide both proofs.
\begin{enumerate} 
\item We apply Corollary \ref{Cor:exist}. Identify the Boolean algebras $\mathcal{L}_{xy}(A)$ and $\mathcal{L}_{yz}(A)$ with their natural embeddings into $\mathcal{L}_{xyz}(A)$ as well as $\lambda_1$ and $\lambda_2$ with the corresponding finitely additive measure on the corresponding image. Suppose that $\theta_1(x,y) \in \mathcal{L}_{xy}(A)$, $\theta_2(y,z) \in \mathcal{L}_{yz}(A)$ and $(\theta_1(x,y) \wedge z =z) \subseteq (\theta_2(y,z) \wedge x = x)$ as definable sets. Then we set $\Gamma(x,y,z) : = \exists x (\theta_1(x,y) \wedge z =z) \wedge x = x$. It follows that $(\theta_1(x,y) \wedge z = z) \subseteq \Gamma(x,y,z)$ and $\Gamma(x,y,z) \subseteq (\theta_2(y,z) \wedge x =x)$. Now, 
\begin{equation*}
    \lambda_1(\theta_1(x,y) \wedge z =z) \leq \lambda_1(\Gamma(x,y,z)) = \pi_{y}(\lambda_1)(\exists x \theta_1(x,y)) 
\end{equation*}
\begin{equation*}
   = \pi_{y}(\lambda_2)(\exists x \theta_1(x,y)) = \lambda_{2}(\Gamma(x,y,z)) \leq \lambda_2(\theta_{2}(y,z) \wedge x = x). 
\end{equation*}
By Corollary \ref{Cor:exist}, there exists $\lambda_3 \in \mathfrak{M}_{xyz}(A)$ with the desired properties. 
\item Again we apply Corollary \ref{Cor:exist}. Consider the Boolean algebras $\mathcal{L}_{xyz}(A)$ and $\mathcal{L}_{xz}(\mathcal{U})$ embedded in $\mathcal{L}_{xyz}(\mathcal{U})$ with the usual identifications on the measures. Suppose that $\theta(x,y,z) \in \mathcal{L}_{xyz}(A)$ and $ \psi(x,z) \in \mathcal{L}_{xz}(\mathcal{U})$ such that $\theta(x,y,z) \subseteq (\psi(x,z) \wedge y = y)$. Let $\Gamma(x,y,z) := \exists y \theta(x,y,z) \wedge y = y$. Again, it follows that $\theta(x,y,z) \subseteq \Gamma(x,y,z) \subseteq (\psi(x,z) \wedge y = y)$. Hence
\begin{equation*} 
\lambda(\theta(x,y,z)) \leq \lambda \left( \Gamma(x,y,z)\right) = \pi_{xz}(\lambda) \left( \exists y \theta(x,y,z) \right)
\end{equation*} 
\begin{equation*} 
= \omega(\exists y(\theta(x,y,z)) = \omega(\Gamma(x,y,z)) \leq \omega \left(\psi(x,y) \wedge y = y \right).
\end{equation*}  
By Corollary \ref{Cor:exist}, there exists $\omega' \in \mathfrak{M}_{xyz}(\mathcal{U})$ with the desired properties. \qedhere
\end{enumerate}
\end{proof}

\begin{proposition}\label{transitive:1} Let $\mu \in \mathfrak{M}_{x}(\mathcal{U})$, $\nu \in \mathfrak{M}_{y}(\mathcal{U})$ and $\eta \in \mathfrak{M}_{z}(\mathcal{U})$.  If $\mu \geq_{\mathbb{E},A} \nu$ and $\nu \geq_{\mathbb{E},A} \eta$, then $\mu \geq_{\mathbb{E},A} \eta$.
\end{proposition}

\begin{proof}
Since $\mu \geq_{\mathbb{E},A} \nu$, there exists $\lambda_1 \in \mathfrak{M}_{xy}(A)$ which witnesses domination. Likewise, since $\nu \geq_{\mathbb{E},A} \mu$, there exists $\lambda_2 \in \mathfrak{M}_{yz}(A)$ which witness domination. Observe that $\pi_{y}(\lambda_1) = \pi_{y}(\lambda_2) = \nu|_{A}$. By Lemma~\ref{ext:lemmas}(1) , there exists $\lambda_{3} \in \mathfrak{M}_{xyz}(A)$ such that $\pi_{xy}(\lambda_{3}) = \lambda_1$ and $\pi_{yz}(\lambda_{3}) = \lambda_2$. 
Let $\lambda := \pi_{xz}(\lambda_3)$. We now show that $\lambda$ witnesses $\mu \geq_{\mathbb{E},A} \eta$. Observe that $\pi_{x}(\lambda) = \mu|_{A}$ and fix $\omega \in \E(\lambda,\mu)$. It suffices to show that $\pi_{z}(\omega) = \eta$. 

\begin{enumerate} 
\item By definition, $\omega|_{A} = \lambda =  \pi_{xz}(\lambda_3)$. By Lemma \ref{ext:lemmas}(2), there exists $\omega' \in \mathfrak{M}_{xyz}(\mathcal{U})$ such that $\pi_{xz}(\omega') = \omega$ and $\omega'|_{A} = \lambda_3$.
\item Consider the measure $\pi_{xy}(\omega')$. Notice that $\pi_{x}(\pi_{xy}(\omega')) = \pi_{x}(\omega) = \mu$ and $\pi_{xy}(\omega')|_{A} = \pi_{xy}(\omega'|_{A}) = \pi_{xy}(\lambda_3) = \lambda_1$. Hence $\pi_{xy}(\omega') \in \E(\lambda_1,\mu)$ and by domination, $\pi_{y}(\pi_{xy}(\omega')) = \nu$. Therefore $\pi_{y}(\omega') = \nu$. 
\item Now consider the measure $\pi_{yz}(\omega')$. By above, $\pi_{y}(\pi_{yz}(\omega')) = \pi_{y}(\omega') = \nu$. Also, $\pi_{yz}(\omega')|_{A} = \pi_{yz}(\omega'|_{A}) = \pi_{yz}(\lambda_3) = \lambda_2$. Hence $\pi_{yz}(\omega') \in \E(\lambda_2,\nu)$ and by domination, $\pi_{z}(\pi_{yz}(\omega')) = \eta$. Therefore $\pi_{z}(\omega) = \pi_{z}(\pi_{yz}(\omega')) = \eta$.
\end{enumerate} 
Thus for any $\omega \in \E(\lambda,\mu)$, $\pi_{z}(\omega) = \eta$ and so $\mu \geq_{\mathbb{E},A} \eta$. 
\end{proof}

\begin{theorem}\label{Thm:pre} The relation $\geq_{\mathbb{E}}$ is a preorder on global Keisler measures. 
\end{theorem} 

\begin{proof} We need to show that $\geq_{\mathbb{E}}$ is transitive and reflexive. 
\begin{enumerate} 
\item Transitive: Suppose that $\mu \geq_{\mathbb{E}} \nu$ and $\nu \geq_{\mathbb{E}} \eta$. Then there exists small sets $A_1$ and $A_2$ such that $\mu \geq_{\mathbb{E},A_1} \nu$ and $\nu \geq_{\mathbb{E},A_2} \eta$. 
By Lemma \ref{lem:ext-para}, we have that $\mu \geq_{\mathbb{E},A_1 \cup A_2} \nu$ 
and $\nu \geq_{\mathbb{E},A_1 \cup A_2} \eta$. By Proposition \ref{transitive:1},  $\mu \geq_{\mathbb{E},A_1 \cup A_2} \eta$ and so $\mu \geq_{\mathbb{E}} \eta$. 
\item Reflexive: Notice that $\mu$ is a push-forward of itself by considering the identity map, i.e. the map $f:\mathcal{U}^{x} \to \mathcal{U}^{y}$ via $f(a) = a$. By Proposition \ref{pushforward}, $\mu \geq_{\mathbb{E},\emptyset} \mu$ and so $\mu \geq_{\mathbb{E}} \mu$. 
\end{enumerate} 
Hence $\geq_{\mathbb{E}}$ is a preorder. 
\end{proof} 

\begin{corollary} Let $\mu \in \mathfrak{M}_{x}(\mathcal{U})$, $\nu \in \mathfrak{M}_{y}(\mathcal{U})$, and $f:\mathcal{U}^{y} \to \mathcal{U}^{z}$ be a definable map. If $\mu \geq_{\mathbb{E}} \nu$ then $\mu \geq_{\mathbb{E}} f(\nu)$.  
\end{corollary}

\begin{proof}
The map $f$ is $B$-definable for some finite set $B$ and so $\nu \geq_{\mathbb{E},B} f(\nu)$ by Proposition \ref{pushforward}. Since $\mu \geq_{\mathbb{E}} \nu$, there exists a small set $A$ such that $\mu \geq_{\mathbb{E},A} \nu$. By Lemma \ref{lem:ext-para}, $\mu \geq_{\mathbb{E},A \cup B} \nu$ and $\nu \geq_{\mathbb{E},A \cup B} f(\nu)$. By transitivity, $\mu \geq_{\mathbb{E},A\cup B} f(\nu)$. 
\end{proof}

\subsection{Approximation results} As in the case of domination for types, if $\mu$ dominates $\nu$ and $\lambda$ witnesses domination, then there is a collection of formulas which \textit{forces} this to occur. This is made precise in terms of the approximation results in this section. Moreover,  this collection of formulas forms a Boolean subalgebra of $\mathcal{L}_{xy}(\mathcal{U})$ which we denote as $\mathbb{B}_{xy}(A)$. We show that for each $\omega$ in $\E(\lambda,\mu)$ and $\epsilon > 0$, there exists formulas $\gamma^{-}(x,y)$ and $\gamma^{+}(x,y)$ in $\mathbb{B}_{xy}(A)$ which \textit{squeeze} the value of $\pi(\omega)(\varphi(y))$. An application of (topological) compactness gives us a uniform version of this result.  We begin with defining the auxiliary Boolean algebra $\mathbb{B}_{xy}(A)$.  

\begin{definition} Let $A \subseteq \mathcal{U}$. We let $\mathbb{B}_{xy}(A)$ be the Boolean subalgebra of $\mathcal{L}_{xy}(\mathcal{U})$ generated by $\{\varphi(x) \wedge y = y: \varphi(x) \in \mathcal{L}_{x}(\mathcal{U}) \} \cup \{\theta(x,y): \theta(x,y) \in \mathcal{L}_{xy}(A)\}$. 
\end{definition}

\begin{remark}\label{Bool:struc} The Boolean algebra $\mathbb{B}_{xy}(A)$ is generated by formulas of the form $\{\theta(x,y) \wedge \psi(x): \theta(x,y) \in \mathcal{L}_{xy}(A), \psi(x) \in \mathcal{L}_{x}(\mathcal{U})\}$ using only negation and intersection. Notice that $\theta(x,y) \vee \psi(x) \equiv \neg (\neg \theta(x,y) \wedge \neg \psi(x))$. 
\end{remark}

\begin{proposition}\label{approx:1} Suppose that $\mu \geq_{\mathbb{E},A} \nu$ and let $\lambda \in \mathfrak{M}_{xy}(A)$ witness domination. Then for any $\omega \in \E(\lambda, \mu)$,  $\varphi(y) \in \mathcal{L}_{y}(\mathcal{U})$, and $\epsilon > 0$, there exists formulas $\gamma^{-}(x,y), \gamma^{+}(x,y) \in \mathbb{B}_{xy}(A)$ such that 
 \begin{enumerate} 
 \item $ \gamma^{-}(x,y) \subseteq (\varphi(y) \wedge x = x) \subseteq \gamma^{+}(x,y) $
 \item $\omega(\gamma^{+}(x,y)) - \omega(\gamma^{-}(x,y)) < \epsilon$.   
 \item $|\omega(\gamma^{\dagger}(x,y)) - \nu(\varphi(y))| < \epsilon$ where $\dagger \in \{+,-\}$. 
 \end{enumerate} 
\end{proposition} 

\begin{proof} Fix $\varphi(y) \in \mathcal{L}_{y}(\mathcal{U})$ and let $r_1 = \sup\{\omega(\gamma(x,y)): \gamma(x,y) \in \mathbb{B}_{xy}(A), \gamma(x,y) \subseteq (\varphi(y) \wedge x = x)\}$ and let $r_2 = \inf\{\omega(\gamma(x,y)): \gamma(x,y) \in \mathbb{B}_{xy}(A), (\varphi(y) \wedge x = x) \subseteq \gamma(x,y)\}$. Clearly $r_1\leq r_2$.

 Now suppose that $r_1 < r_2$. By Fact \ref{fact:charges-r}, for any $r \in [r_1,r_2]$, there exists a measure $\omega_{r} \in \mathfrak{M}_{xy}(\mathcal{U})$ such that $\omega_{r} \supseteq \omega|_{\mathbb{B}_{xy}(A)}$ and $\omega_{r}(\varphi(y) \wedge x =x) = r$. Choose $r_{*}$ such that $\omega_{r_*}(\varphi(y) \wedge x = x) \neq \nu(\varphi(y))$. However, $\omega_{r_*}|_{\mathbb{B}_{xy}(A)} = \omega|_{\mathbb{B}_{xy}(A)}$ which implies that $\omega_{r} \in \E(\lambda,\mu)$. By domination, it follows that $\pi_{y}(\omega_{r_{*}})(\varphi(y) \wedge x = x) = \nu(\varphi(y))$, a contradiction. Therefore $r_1 = r_2 = \nu(\varphi(y))$.

Fix $\epsilon>0$. Since $r_1 = r_2 = \nu(\varphi(y))$, there exists formulas  $\gamma^{-}(x,y)$ and $\gamma^{+}(x,y)$ in  $\mathbb{B}_{xy}(A)$ such that
\begin{enumerate}
    \item[(i)] $\gamma^{-}(x,y) \subseteq (\varphi(y) \wedge x = x)$ and $\mid\omega(\gamma^{-}(x,y)) - \nu(\varphi(y))\mid <\epsilon/2$. 
    \item[(ii)] $ (\varphi(y) \wedge x = x) \subseteq \gamma^{+}(x,y)$ and $\mid \omega(\gamma^{+}(x,y))-\nu(\varphi(y))\mid <\epsilon/2$.
\end{enumerate}
Statements (1) and (3) have been satisfied while (2) follows from a basic application of the triangle inequality. 
\end{proof}

The following theorem is a uniform version of the previous proposition. 

\begin{theorem}\label{approx:2} Suppose that $\mu \geq_{\mathbb{E},A} \nu$ and $\lambda$ witnesses domination. Then for every formula $\psi(y) \in \mathcal{L}_{y}(\mathcal{U})$ and $\epsilon > 0$ there exists formulas $\chi_{\epsilon}^{-}(x,y)$ and $\chi_{\epsilon}^{+}(x,y)$ in $\mathbb{B}_{xy}(A)$ such that
\begin{enumerate}
\item $\chi_{\epsilon}^{-}(x,y) \subseteq (\psi(y) \wedge x = x) \subseteq \chi_{\epsilon}^{+}(x,y) $. 
\item for any $\omega \in \E(\lambda,\mu)$,  $\omega(\chi_{\epsilon}^{+}(x,y)) - \omega(\chi_{\epsilon}^{-}(x,y)) < \epsilon$.
\item for any $\omega \in \E(\lambda, \mu)$, $|\omega(\chi_{\epsilon}^{\dagger}(x,y)) - \nu(\varphi(y))| < \epsilon$ where $\dagger \in \{+,-\}$. 
\end{enumerate}
\end{theorem}

\begin{proof} Fix $\psi(y) \in \mathcal{L}_{y}(\mathcal{U})$ and  $\epsilon > 0$. By Lemma \ref{ext:exist}, the space $\E(\lambda,\mu)$ is a non-empty compact Hausdorff space. By Proposition \ref{approx:1}, for each $\omega$ in  $\E(\lambda,\mu)$, there exists formulas $\gamma_{\omega}^{-}(x,y)$ and $\gamma_{\omega}^{+}(x,y)$ in $\mathbb{B}_{xy}(A)$ such that 
\begin{enumerate}
\item $\gamma_{\omega}^{-}(x,y) \subseteq (\varphi(y) \wedge x = x) \subseteq  \gamma_{\omega}^{+}(x,y)$. 
\item $\omega(\gamma_{\omega}^{+}(x,y)) - \omega(\gamma_{\omega}^{-}(x,y)) < \epsilon$. 
\end{enumerate} 
Let $D_{\omega}:\mathfrak{M}_{xy}(\mathcal{U}) \to [0,1]$, $D_{\omega}(\eta) = \eta(\gamma_{\omega}^{+}(x,y)) - \eta(\gamma_{\omega}^{-}(x,y))$. Since evaluation maps are continuous, $D_{\omega}$ is continuous. Let $A_{\omega} = D_{\omega}^{-1}([0,\epsilon))$. Now $\bigcup_{\omega \in \E(\lambda,\mu)} A_{\omega}$ is an open cover of $\E(\lambda,\mu)$ since each $\omega$ in $\E(\lambda,\mu)$ is in $A_{\omega}$. Hence there is a finite subcover. Without loss of generality suppose that $\bigcup_{i=1}^{n} A_{\omega_{i}}$ is a finite subcover and consider the formula $\chi_{\epsilon}^{-}(x,y) := \bigvee_{i=1}^{n} \gamma_{i}^{-}(x,y)$ and $\chi_{\epsilon}^{+}(x,y):=\bigwedge_{i=1}^{n} \gamma_{i}^{+}(x,y)$. Clearly, $\chi_{\epsilon}^{-}(x,y) \subseteq (\varphi(y) \wedge x= x) \subseteq \chi_{\epsilon}^{+}(x,y)$.  

Fix $\omega \in \E(\lambda,\mu)$. Then $\omega \in A_{\omega_i}$ for some $i \leq n$ and so we have that 
\begin{equation*} 
|\omega(\chi_{\epsilon}^{+}(x,y)) - \omega(\chi_{\epsilon}^{-}(x,y)) |\leq |\omega(\gamma_{i}^{+}(x,y)) - \omega(\gamma_{i}^{-}(x,y))|=D_{\omega_i}(\omega) < \epsilon. 
\end{equation*} 
Statement (3) follows from an application of the triangle inequality. 
\end{proof}

\subsection{Localizations}

Localizations are inherently measure theoretic constructions which are trivial for types. We show that if $\mu$ dominates $\nu$ and if this domination is witnessed by a separated amalgam (e.g. the restriction of a Morley product), then $\mu$ dominates any localization of $\nu$. This phenomenon also occurs when $\mu = \delta_p$ for some type $p$ and $\nu$ is any measure.  We begin with the definition of localization of a measure \cite[Definition 1.5]{NIP-III}.  

\begin{definition} Let $\nu \in \mathfrak{M}_{y}(A)$ and $\varphi(y) \in \mathcal{L}_{y}(A)$. Suppose that $0< \nu(\varphi(y))$. Let $\nu_{[\varphi]}$ be the \emph{localization of $\nu$ to $\varphi(y)$}, i.e. for any $\psi(y) \in \mathcal{L}_{y}(A)$,
\begin{equation*}
\nu_{[\varphi]}(\psi(y)) = \frac{\nu(\varphi(y) \wedge \psi(y))}{\nu(\varphi(y))}. 
\end{equation*}
\end{definition}

\begin{proposition}\label{local:ext} Let $\mu \in \mathfrak{M}_{x}(\mathcal{U})$, $\nu \in \mathfrak{M}_{y}(\mathcal{U})$, and $\varphi(y) \in \mathcal{L}_{y}(A)$. Suppose that  $0< \nu(\varphi(y)) < 1$, $\mu \geq_{\mathbb{E},A} \nu$ and $\lambda \in \mathfrak{M}^{\am}_{xy}(A)$ witnesses domination. Then $\mu \geq_{\mathbb{E},A} \nu_{[\varphi]}$. 
\end{proposition} 

\begin{proof} Let $r := \nu(\varphi(y))$ and $s:= \nu(\neg \varphi(y))$. Let $\Phi(x,y) := (\varphi(y) \wedge x = x)$ and  consider the measures $\lambda_{[\Phi]}$ and $\lambda_{[\neg \Phi]}$ in $\mathfrak{M}_{xy}(A)$. We show $\lambda_{[\Phi]}$ witness $\mu \geq_{\mathbb{E},A} \nu_{[\varphi]}$.

\begin{enumerate} 
\item We claim that $\pi_{x}(\lambda_{[\Phi]}) = \mu|_{A}$. Fix $\psi(x) \in \mathcal{L}_{x}(A)$ and notice 
\begin{equation*} 
\pi_{x}(\lambda_{[\Phi]})(\psi(x)) = \frac{\lambda(\psi(x) \wedge \varphi(y))}{\lambda(\varphi(y) \wedge x=x)} = \frac{\pi_{x}(\lambda(\psi(x))) \cdot \pi_{y}(\lambda(\varphi(y)))}{\pi_{y}(\lambda)(\varphi(y))} 
\end{equation*} 
\begin{equation*} 
= \pi_{x}(\lambda)(\psi(x)) = \mu(\psi(x)). 
\end{equation*} 
Notice that the second equality follows from the fact that $\lambda \in \mathfrak{M}^{\am}_{xy}(A)$ and so this assumption is necessary in our proof. 
\item By (1), we have some $\omega \in \E(\lambda_{[\Phi]},\mu)$. It suffices to prove that $\pi_{y}(\omega) = \nu_{[\varphi]}$. 
\item Similarly, $\pi_{x}(\lambda_{[\neg \Phi]}) = \mu|_{A}$. Hence $\E(\lambda_{[\neg \Phi]},\mu) \neq \emptyset$ and let $\eta \in \E(\lambda_{[\neg \Phi]},\mu)$. 
\item Let $\omega' = r\omega + s\eta$. We show that $\omega' \in \E(\lambda, \mu)$. 
\begin{enumerate} 
\item Claim: $\omega'|_{A} = \lambda$. Let $\theta(x,y) \in \mathcal{L}_{xy}(A)$. Then 
\begin{equation*}
 \omega'(\theta(x,y)) = (r\lambda_{[\Phi]} + s\lambda_{[\neg \Phi]})(\theta(x,y)) 
\end{equation*}
\begin{equation*}
     =  \frac{r\lambda(\theta(x,y) \wedge \varphi(y))}{r} + \frac{s\lambda(\theta(x,y) \wedge \neg \varphi(y))}{s} = \lambda(\theta(x,y)). 
\end{equation*}
\item Claim: $\pi_{x}(\omega') = \mu$. Since $\omega \in \E(\lambda_{[\Phi]},\mu)$ and $\eta \in  \E(\lambda_{[\neg\Phi]},\mu)$,
\begin{equation*}
    \pi_{x}(\omega')= \pi_{x}((r\omega + s\eta)) = r\pi_{x}(\omega) + s\pi_{x}(\eta) = r\mu + s\mu = \mu. 
\end{equation*}
\end{enumerate} 
\item Since $\mu \geq_{\mathbb{E},A} \nu$ and $\lambda$ witnesses domination, we conclude that $\pi_{y}(\omega') = \nu$. 
\item Finally, we argue that $\pi_{y}(\omega) = \nu_{[\varphi]}$. Fix $\rho(y) \in \mathcal{L}_{y}(\mathcal{U})$ and consider the following sequence of computations: 
\begin{equation*}
   \pi_{y}(\omega)(\rho(y)) = \omega(\rho(y) \wedge x = x) \overset{(a)}{=} \omega(\rho(y) \wedge \varphi(y) \wedge x = x)
\end{equation*}
\begin{equation*} 
\overset{(b)}{=} \frac{r\omega(\rho(y) \wedge \varphi(y) \wedge x = x)}{r \omega(\varphi(y) \wedge x = x)}
\end{equation*} 
\begin{equation*}
\overset{(c)}{=} \frac{r\omega(\rho(y) \wedge \varphi(y) \wedge x = x) + s\eta(\rho(y) \wedge \varphi(y) \wedge x = x))}{r\omega(\varphi(y) \wedge x = x) + s\eta(\varphi(y) \wedge x = x)} 
\end{equation*} 
\begin{equation*}
   = \frac{\omega'(\rho(y) \wedge \varphi(y) \wedge x = x)}{\omega'(\varphi(y) \wedge x = x)} 
\end{equation*}
\begin{equation*}
    = \frac{\pi_{y}(\omega')(\rho(y) \wedge \varphi(y))}{\pi_{y}(\omega')(\varphi(y))} \overset{(d)}{=} \frac{\nu(\rho(y) \wedge \varphi(y))}{\nu(\varphi(y))} = \nu_{[\varphi]}(\rho(y)). 
\end{equation*}
We now provide the following justifications: 
\begin{enumerate}[(a)]
\item Follows from $1= \lambda_{[\varphi]}(\varphi(y) \wedge x = x) = \omega(\varphi(y) \wedge x = x)$. 
\item Since $\omega(\varphi(y) \wedge x = x) = 1$, we are dividing the numerator and denominator by $r$. 
\item Follows from $\eta(\varphi(y) \wedge x = x) = 0$.
\item Follows from domination, i.e  $\pi_{y}(\omega') = \nu$. \qedhere
\end{enumerate} 
\end{enumerate}
\end{proof}

\begin{corollary} Suppose that $\delta_{p} \geq_{\mathbb{E},A} \nu$ and $0 <\nu(\varphi(y)) < 1$ for some $\varphi(y)\in \mathcal{L}(A)$. Then $\delta_{p} \geq_{\mathbb{E},A} \nu_{[\varphi]}$. 
\end{corollary} 

\begin{proof} Let $\lambda$ witness $\delta_{p} \geq_{\mathbb{E},A} \nu$. Then $\lambda \in \mathfrak{M}_{xy}^{\am}(A)$ (see Lemma \ref{type:amal}) and so by Proposition \ref{local:ext}, the statement holds. 

\end{proof}

\section{Basic Preservation Results}\label{sec:preserve}

In this section, we discuss several preservation theorems. For classical domination, realized types form the bottom tier of the domination preorder, and notions like invariance, finite satisfiability, and definability  are preserved under domination (See ~\cite{rosario-inv-dom}). For extension domination, we show that the collection of smooth measures form the bottom tier of the extension domination preorder and if $\mu \geq_{\mathbb{E}} \nu$ and $\mu$ is invariant, smooth, or finitely satisfiable, then $\nu$ is respectively invariant, smooth, or finitely satisfiable. For definable measures, the preservation remains open.

\subsection{Invariance} We first show that invariance is preserved under extension domination. We recall the definition. 

\begin{definition} We say that $\mu\in \mathfrak{M}_x(\mathcal{U})$ is \emph{invariant over $A$} or \emph{$A$-invariant} if for every formula $\varphi(x,z) \in \mathcal{L}_{xz}(\emptyset)$ and $b,c \in \mathcal{U}^{z}$ such that $\tp(b/A) = \tp(c/A)$, we have $\mu(\varphi(x,b)) = \mu(\varphi(x,c))$. We use $\mathfrak{M}^{\inv}_x(\mathcal{U},A)$ to denote the set of all $A$-invariant measures.

We recall that if $\sigma \in \Aut(\mathcal{U},A)$, then for any variable $x$, $\sigma$ induces a map $\sigma_{*}:\mathfrak{M}_{x}(\mathcal{U}) \to \mathfrak{M}_{x}(\mathcal{U})$ via $\sigma_{*}(\omega) = \omega_{\sigma}$ where $\omega_{\sigma}(\theta(x,b)) = \omega(\theta(x,\sigma^{-1}(b))$. We remark that $\omega$ is \textit{$A$-invariant} if and only if $\omega=\omega_\sigma$ for all $\sigma\in \Aut(\mathcal{U},A)$.
\end{definition} 

\begin{lemma}\label{aut:affine} Assume that $\mu \in \mathfrak{M}^{\inv}_{x}(\mathcal{U}, A)$, $\lambda \in \mathfrak{M}_{xy}(A)$, $\mu|_{A} = \pi_{x}(\lambda)$, and $\sigma \in \Aut(\mathcal{U},A)$. Then the map $\sigma_{*}|_{\E(\lambda,\mu)}: \E(\lambda,\mu) \to \E(\lambda,\mu) $ is an affine homeomorphism. Moreover, for every $\sigma \in \Aut(\mathcal{U},A)$, there exists a measure $\omega \in \E(\lambda,\mu)$ such that $\omega_\sigma= \omega$.
\end{lemma}
\begin{proof} Fix $\sigma \in \Aut(\mathcal{U},A)$. 
\begin{enumerate} 
\item We first show that $\E(\lambda,\mu)$ is invariant under $\sigma_*$. Fix $\omega \in \E(\lambda,\mu)$. Notice that $\omega_{\sigma}|_{A} = \lambda$ since $\sigma$ fixes all formulas with parameters only from $A$. Also, $\pi_{x}(\omega_{\sigma}) = \mu$ by the invariance of $\mu$ over $A$.
\item Injective: Suppose that $\omega \neq \eta$. Then there exists a formula $\varphi(x,y,b) \in \mathcal{L}_{xy}(\mathcal{U})$ such that $\omega(\varphi(x,y,b)) \neq \eta(\varphi(x,y,b))$. Then $\omega_\sigma(\varphi(x,y,\sigma(b)) \neq \eta_\sigma(\varphi(x,y,\sigma(b))$. 
\item Surjective: Fix $\omega \in \E(\lambda,\mu)$. By (1), $\omega_{\sigma^{-1}} \in \E(\lambda,\mu)$. Then $\sigma_{*}(\omega_{\sigma^{-1}}) = \omega$. 
\item Affine homeomorphism: Clearly $\sigma_*$ preserves convex combinations and by above, $\sigma_{*}(\E(\lambda,\mu)) = \E(\lambda,\mu)$. It suffices to show that $\sigma_*$ is a homeomorphism from $\mathfrak{M}_{xy}(\mathcal{U})$ to $\mathfrak{M}_{xy}(\mathcal{U})$. The map $\sigma_*$ is continuous since the inverse image of a basic open set is basic open (see Fact \ref{Basic:Keisler}). Hence $\sigma_{*}$ is a continuous bijeciton between compact Hausdorff spaces and therefore a homeomorphism. 
\item Moreover portion: By Lemma~\ref{ext:exist}, $\E(\lambda,\mu)$ is a convex compact space. By above, $\sigma_{*}|_{\E(\lambda,\mu)}: \E(\lambda,\mu) \to \E(\lambda,\mu)$ is an affine homeomorphism. Hence an application of Schauder fixed point theorem establishes the claim.  \qedhere
\end{enumerate} 
\end{proof} 
\begin{proposition}\label{pres:inv} Let $\mu \in \mathfrak{M}_{x}(\mathcal{U})$ and $\nu \in \mathfrak{M}_{y}(\mathcal{U})$. Assume that $\mu$ is invariant over $A$ and $\mu \geq_{\mathbb{E},A} \nu$. Then $\nu$ is invariant over $A$. Moreover, if $\mu$ is invariant and $\mu \geq_{\mathbb{E}} \nu$, then $\nu$ is invariant. 
\end{proposition} 

\begin{proof} Fix $\varphi(y,z) \in \mathcal{L}_{xy}(\emptyset)$ and suppose that $b, c \in \mathcal{U}^{z}$ with $\tp(b/A) = \tp(c/A)$. Choose $\sigma \in \Aut(\mathcal{U}/A)$ such that $\sigma(b) = c$. Since $\mu \geq_{\mathbb{E},A} \nu$, there exists some $\lambda \in \mathfrak{M}_{xy}(A)$ witnessing domination. By Lemma \ref{aut:affine}, take $\omega \in \E(\lambda,\mu)$ such that $\omega_\sigma=\omega$. By domination, $\pi_{y}(\omega) = \nu$. Thus $\nu(\varphi(y,c))=\omega(\varphi(y,c)\wedge x=x)=\omega(\varphi(y,b)\wedge x=x)=\nu(\varphi(y,b))$, proving the $A$-invariance of $\nu$.

The moreover portion follows from the previous paragraph, the  and an application of Lemma \ref{lem:ext-para}.
\end{proof} 

\subsection{Smoothness} We show that the collection of smooth measures form the bottom tier of the extension domination preorder and smoothness is preserved under extension domination. It is interesting to note that the proofs of these propositions follow from the usual definition of smoothness (i.e. unique extensions) and not the associated \textit{squeeze property}. We recall the definition of smooth \cite[Definition 2.5]{hrushovski2008groups}, \cite[Definition 2.1]{NIP-III}.

\begin{definition} Let $\mu \in \mathfrak{M}_{x}(\mathcal{U})$. We say that $\mu$ is \textit{smooth} over $A$ if for any $\nu \in \mathfrak{M}_{x}(\mathcal{U})$ such that $\nu|_{A} = \mu|_{A}$, then $\nu = \mu$. We say $\mu$ is \textit{smooth} if it is smooth over some small $A$.

Moreover, if $\eta \in \mathfrak{M}_{x}(A)$, we say that $\eta$ is smooth over $A$ if there exists a unique measure $\mu \in \mathfrak{M}_{x}(\mathcal{U})$ such that $\mu|_{A} = \eta$. We call $\mu$ the \textit{unique extension} of $\eta$. 
\end{definition} 
Note that for a global type $p$, $\delta_p$ is smooth if and only if $p$ is realized. 

\begin{proposition}\label{smooth:all} Let $\nu \in \mathfrak{M}_{y}(\mathcal{U})$, $A \subseteq \mathcal{U}$, and suppose that $\nu$ is smooth over $A$. Then for any $\mu \in \mathfrak{M}_{x}(\mathcal{U})$,  $\mu \geq_{\mathbb{E},A} \nu$. 
\end{proposition}

\begin{proof} Choose any $\lambda$ in $\mathfrak{M}_{xy}(A)$ such that $\pi_{x}(\lambda) = \mu|_A$ and $\pi_{y}(\lambda) = \nu|_A$. By Remark \ref{samalgam:exist}, such $\lambda$ exists. Let $\omega \in \E(\lambda, \mu)$. Notice that 
\begin{equation*} \pi_{y}(\omega)|_{A} = \pi_{y}(\omega|_{A}) = \pi_{y}(\lambda) = \nu|_{A}. 
\end{equation*} 
 Hence $\pi_{y}(\omega)$ extends $\nu|_{A}$. Since $\nu$ is smooth over $A$, it follows that $\pi_{y}(\omega) = \nu$. 
\end{proof}

\begin{proposition}\label{smooth:pres} Suppose that $\mu$ is smooth over $A$ and $\mu \geq_{\mathbb{E},A} \nu$. Then $\nu$ is smooth over $A$. 
\end{proposition}

\begin{proof} Let $\lambda$ witness $\mu \geq_{\mathbb{E},A} \nu$. For any $\omega \in \E(\lambda,\mu)$, $\pi_{y}(\omega) = \nu$. Suppose that there exists $\nu_1$ and $\nu_2$ in $\mathfrak{M}_{y}(\mathcal{U})$ such that $\nu_1 \neq \nu_2$ and $\nu_1|_{A} = \nu_2|_{A} = \nu|_{A}$. By Lemma \ref{ext:exist}, the sets $\E(\lambda,\nu_1)$ and $\E(\lambda,\nu_2)$ are non-empty and so let $\omega_i \in \E(\lambda,\nu_i)$ for $i \in \{1,2\}$. 

We now argue that $\omega_1,\omega_2 \in \E(\lambda,\mu)$. By construction, $\omega_i|_{A} = \lambda$. Now notice that $\pi_{x}(\omega_i)|_{A} = \pi_{x}(\omega_i|_{A})  = \pi_{x}(\lambda) = \mu|_{A}$. Since $\mu$ is smooth over $A$, we have that $\pi_{x}(\omega_i) = \mu$. Hence $\omega_1,\omega_2 \in \E(\lambda,\mu)$. By domination (i.e. $\mu \geq_{\mathbb{E},A} \nu)$, we have that 
\begin{equation*} \nu_1 = \pi_{y}(\omega_1) = \nu = \pi_{y}(\omega_2) = \nu_2,
\end{equation*}
which completes the proof. 
\end{proof}

\begin{proposition}\label{smooth:2} Suppose that $\mu \geq_{\mathbb{E}} \nu$ and $\mu$ is smooth. Then $\nu$ is smooth. 
\end{proposition} 

\begin{proof} Since $\mu$ is smooth, there exists some small set $A$ such that $\mu$ is smooth over $A$. It is clear from the definition that if $C \supseteq A$, then $\mu$ is also smooth over $C$. Since $\mu$ dominates $\nu$, there exists some small set $B$ such that $\mu \geq_{\mathbb{E},B} \nu$. By Lemma~\ref{lem:ext-para}, $\mu \geq_{\mathbb{E},A \cup B} \nu$. By the previous proposition, $\nu$ is smooth over $A \cup B$. 
\end{proof} 

The next proposition shows that if the measure \textit{witnessing domination} is smooth, then both measures involved are smooth. 

\begin{proposition} Suppose that $\mu \geq_{\mathbb{E},A} \nu$ and $\lambda$ witnesses domination. If $\lambda$ is smooth over $A$ then both $\mu$ and $\nu$ are smooth over $A$.  
\end{proposition} 

\begin{proof} Suppose that $\mu$ is not smooth over $A$. Then there exists distinct $\mu_1, \mu_2 \in \mathfrak{M}_{x}(\mathcal{U})$ such that $\mu_1|_{A} = \mu_2|_{A} = \mu|_{A}$. By Lemma  \ref{ext:exist}, the spaces $\E(\lambda,\mu_1)$ and $\E(\lambda,\mu_2)$ are non-empty. Let $\omega_1 \in \E(\lambda,\mu_1)$ an $\omega_2 \in \E(\lambda,\mu_2)$. Notice that $\omega_1 \neq \omega_2$ but $\omega_1|_{A} = \omega_2|_A = \lambda$, a contradiction. A similar argument demonstrates that $\nu$ is smooth over $A$. 
\end{proof}

\subsection{Finite Satisfiability}
In this section, we show that finitely satisfiability is preserved under extension domination. Unlike the proofs of the previous two preservation results, the proof in this case follows from the approximation lemma derived in the first section of the paper (i.e. Proposition \ref{approx:1}). We recall the definition \cite[Definition 2.5]{hrushovski2008groups}. 

\begin{definition} Suppose that $M \prec \mathcal{U}$ and $\mu \in \mathfrak{M}_{x}(\mathcal{U})$. We say that $\mu$ is \textit{finitely satisfiable} in $M$ if for any $\varphi(x)$ in $\mathcal{L}_{x}(\mathcal{U})$, if $\mu(\varphi(x)) > 0$ then there exists some $a \in M$ such that $\mathcal{U} \models \varphi(a)$. We say that $\mu$ is \textit{finitely satisfiable} if it is finitely satisfiable over some small $M\prec \mathcal{U}$
\end{definition}

\begin{lemma}\label{pres:finsat} Let $M \prec \mathcal{U}$ be a small submodel. Suppose that $\mu \in \mathfrak{M}_{x}(\mathcal{U})$, $\mu$ is finitely satisfiable in $M$, $\lambda \in \mathfrak{M}_{xy}(M)$ and $\pi_{x}(\lambda) = \mu|_M$. Let $\omega \in \E(\lambda, \mu)$. Then for any $\gamma(x,y) \in \mathbb{B}_{xy}(M)$, if $\omega(\gamma(x,y)) > 0$, then there exists $(a_0,b_0) \in M^{xy}$ such that $\mathcal{U} \models \gamma(a_0,b_0)$. 
\end{lemma} 

\begin{proof} 
Let $\gamma(x,y)$ be given such that $\omega(\gamma(x,y))>0$. Since $\pi_x(\omega)=\mu$, we have that $\mu(\exists y \gamma(x,y))=\omega(\exists y\gamma(x,y)\wedge y=y)\geq\omega(\gamma(x,y))>0$. Since $\mu$ is finitely satisfiable in $M$, there exists some $a_0\in M$ such that $\mathcal{U} \models \exists y\gamma(a_0,y)$. Consider the set $B=\{b\in \mathcal{U}^y: \models \gamma(a_0,b)\}$. The set $B$ is a non-empty $\mathcal{U}$-definable subset of $\mathcal{U}$. We will argue that $B$ is $M$-invariant, hence $M$-definable, and thus $B \cap M^{y} \neq \emptyset$. By the definition of the Boolean algebra $\mathbb{B}_{xy}(M)$,
 \begin{equation*} \gamma(a_0,y) \equiv \bigvee_{i=1}^{n} \bigwedge_{j_i=1}^{m_i} (\theta_{j_i}(a_0,y) \wedge \psi_{j_i}(a_0))^{\varepsilon_{i}(j_{i})},
\end{equation*} 
where $\varepsilon_{i}: \{1,...,m_i\} \to \{0,1\}$. For $b\in \mathcal{U}^y$ and $\sigma\in \Aut(\mathcal{U}, M)$, $\mathcal{U} \models \theta_{j_{i}}(a_0,b)$ if and only if  $\mathcal{U} \models \theta_i(a_0,\sigma(b))$ since all the parameters which occur in $\theta_i(a_0,y)$ are in $M$. Moreover, $\sigma$ does not affect $\mathcal{U} \models \psi_i(a_0)$ since $\sigma(a_0)=a_0$. Thus $\mathcal{U} \models \gamma(a_0,b)$ if and only if $\mathcal{U} \models \gamma(a_0,\sigma(b))$ and so $B$ is $M$-invariant. Choose $b_0 \in B \cap M^{y}$ and so $\mathcal{U} \models \gamma(a_0,b_0)$. This completes the proof. 
\end{proof}

\begin{theorem} Suppose that $M \prec \mathcal{U}$,  $\mu$ is finitely satisfiable in $M$, and $\mu \geq_{\mathbb{E},M} \nu$. Then $\nu$ is finitely satisfiable in $M$. 

Moreover, if $\mu$ is finitely satisfiable and $\mu \geq_{\mathbb{E}} \nu$, then $\nu$ is finitely satisfiable. 
\end{theorem} 

\begin{proof} Fix $\lambda$ in $\mathfrak{M}_{xy}(M)$ such that $\lambda$ witnesses $\mu \geq_{\mathbb{E},M} \nu$. Fix a formula $\varphi(y) \in \mathcal{L}_{y}(\mathcal{U})$ such that $\nu(\varphi(y)) > 0$. Choose $\epsilon > 0$ such that $\nu(\varphi(y)) > \epsilon$. Fix $\omega$ in $\E(\lambda,\mu)$. By Proposition \ref{approx:1}, there exists a formula $ \gamma(x,y) \in \mathbb{B}_{xy}(M)$ such that 
\begin{enumerate} 
\item  $\gamma(x,y) \subseteq (\varphi(y) \wedge x = x)$
\item $|\omega(\gamma(x,y)) - \omega(\varphi(y) \wedge x =x )| < \frac{\epsilon}{2}$. 
\end{enumerate} 
Notice that 
\begin{equation*}  \nu(\varphi(y)) = \pi_{y}(\omega)(\varphi(y)) =  \omega(\varphi(y) \wedge x = x) \approx_{\frac{\epsilon}{2}} \omega(\gamma(x,y)) \implies \omega(\gamma(x,y)) > \frac{\epsilon}{2}. 
\end{equation*} 
Since $\gamma(x,y) \in \mathbb{B}_{xy}(M)$, we can apply Lemma \ref{pres:finsat} and so there exists $(a,b) \in M^{xy}$ such that $\mathcal{U} \models \gamma(a,b)$. By construction, $\mathcal{U} \models \varphi(b)$ which completes the proof.

The moreover portion follows from the previous paragraph and an application of Lemma \ref{lem:ext-para}.
\end{proof}
\begin{definition} Let $\mu \in \mathfrak{M}_{x}(\mathcal{U})$. Then we say that $\mu$ is \textit{definable} if there exists a small subset $A$ of $\mathcal{U}$ such that $\mu$ is $A$-invariant and for any formula $\varphi(x,z) \in \mathcal{L}_{xz}(\emptyset)$ the map $F_{\mu,A}^{\varphi}: S_{z}(A) \to [0,1]$ via $F_{\mu,A}^{\varphi}(q) = \mu(\varphi(x,b))$ where $b \models q$ is continuous. 
\end{definition}
\begin{question} Suppose that $\mu \geq_{\mathbb{E},A} \nu$ and $\mu$ is $A$-definable. Does this imply that $\nu$ is $A$-definable? Even if $\mu$ or $\nu$ is a type? 
\end{question} 

\section{Separated Amalgams }\label{sec:amalgam}


In this section, we consider a variant of extension domination which we call \textit{amalgam deciphering}. Instead of requiring $\mu$ to dominate $\nu$ over the class of all extension (i.e. $\E(\lambda,\mu)$), we only require that $\mu$ dominates $\nu$ over the class of all separated amalgams, a space we call $\amal(\lambda,\mu)$. This variant is well-behaved with respect to localizations and certain kinds of convex combinations. It is thus tempting to call this property \textit{amalgam domination}. However, fundamental properties about this definition remain open and it seems more than plausible to the authors that this variant fails to be both reflexive and transitive. Therefore, we have decided to use the term \textit{deciphering} in place of \textit{domination}. 

Nevertheless, amalgam deciphering provides some new results about extension domination. We observe that $\delta_{p}$ extension dominates a measure if and only if $\delta_{p}$ amalgam deciphers a measure. From this observation, we show that extension domination is well-behaved with respect to certain convex combination if the measure on the left is a Dirac measure concentrating on a type. 

\subsection{Basic definitions and observations} The difference between extension domination and amalgam deciphering is the collection of measures over which domination occurs.  Similar to extension domination, we begin with the definition of an \textit{amalgam space}.
\begin{definition}\label{def:amal-spa} Let $\mu \in \mathfrak{M}_{x}(\mathcal{U})$ and $\lambda \in \mathfrak{M}^{\am}_{xy}(A)$ such that $\mu|_{A} = \pi_{x}(\lambda)$. We define the corresponding \textit{amalgam space} as follows
\begin{equation*}\amal(\lambda,\mu) := \{\omega \in \mathfrak{M}^{\am}_{xy}(\mathcal{U}): \omega|_{A} = \lambda, \pi_{x}(\omega) = \mu\}.
\end{equation*} 
\end{definition} 

We remark that in the previous definition, we assume that $\lambda \in \mathfrak{M}_{xy}^{\am}(A)$. If we did not, then $\amal(\lambda,\mu)$ would be empty automatically. However, we will see that this is not sufficient to guarantee that $\amal(\lambda,\mu)$ is non-empty (see Example \ref{counter:exist}). 

\begin{lemma}\label{amal:exist} Let $\mu \in \mathfrak{M}_{x}(\mathcal{U})$ and $\lambda \in \mathfrak{M}^{\am}_{xy}(A)$ such that $\mu|_{A} = \pi_{x}(\lambda)$. Then
\begin{enumerate} 
\item $\amal(\lambda,\mu) \subseteq \E(\lambda,\mu)$. 
\item $\amal(\lambda,\mu)$ is a compact Hausdorff space (under the induced topology).  
\item $\amal(\lambda, \mu)$ is convex. 
\end{enumerate} 
\end{lemma}
\begin{proof} First note that if $\amal(\lambda,\mu) = \emptyset$, then the statements are vaccuously true.
\begin{enumerate} 
\item Immediate by definition. 
\item By (1) and Lemma~\ref{ext:exist}, it suffices to show that $\amal(\lambda,\mu)$ is a closed subset of $\E(\lambda,\mu)$. For any $\theta(x,y) \in \mathcal{L}_{xy}(\mathcal{U})$, the map $D_{\theta}:\mathfrak{M}_{xy}(\mathcal{U}) \to [0,1]$ via $D_{\theta}(\omega) = \omega(\theta(x,y))$ is continuous. For each $\varphi(x) \in \mathcal{L}_{x}(\mathcal{U})$ and $\psi(y) \in \mathcal{L}_{y}(\mathcal{U})$, observe that
\begin{equation*} A_{\varphi,\psi} :=\{\omega\in \E(\lambda,\mu): \omega(\varphi(x)\wedge \psi(y))=\omega(\varphi(x)\wedge y=y)\omega(\psi(y)\wedge x=x)\}
\end{equation*}
\begin{equation*}
    =(D_{\varphi \wedge \psi} - D_{\varphi}\cdot D_{\psi})^{-1}(\{0\}).
\end{equation*}
Hence $A_{\varphi,\psi}$ is a closed subset of $\E(\lambda,\mu)$. Notice that $\amal(\lambda,\mu)=\bigcap_{\varphi,\psi}A_{\varphi,\psi}$, and so it is also closed. 
\item Suppose that $\omega_1, \omega_2 \in \amal(\lambda,\mu)$ and fix $r,s \in [0,1]$ such that $r + s = 1$. A straightforward computation shows $\pi_{x}(r\omega_1 + s\omega_2) = \mu$ and $(r\omega_1 + s\omega_2)|_{A} = \lambda$. It remains to show that $r\omega_1 + s\omega_2 \in \mathfrak{M}_{xy}^{\am}(\mathcal{U})$. A direct computation shows that for any $\varphi(x) \in \mathcal{L}_{x}(\mathcal{U})$ and $\psi(y) \in \mathcal{L}_{y}(\mathcal{U})$, 
\begin{equation*} 
(r\omega_1 + s\omega_2)(\varphi(x) \wedge \psi(y)) = \mu(\varphi(x)) \pi_{y}(r\omega_1 + s\omega_2)(\psi(y))
\end{equation*} 
Hence $r\omega_1 + s\omega_2$ is a separated amalgam. \qedhere
\end{enumerate} 
\end{proof}
We now provide an example where $\lambda \in \mathfrak{M}_{xy}^{\am}(A)$ and $\mu|_{A} = \pi_{x}(\lambda)$ and yet $\amal(\lambda,\mu) = \emptyset$. This is in contrast to (1) of  Lemma \ref{ext:exist}. In general, finding whether or not separated amalgams with certain properties exists (as in Lemma \ref{ext:lemmas}) is the main obstruction to the general theory of amalgam deciphering.  

\begin{definition}[Diagonal measure construction] Let $\mu \in \mathfrak{M}_{x}(\mathcal{U})$ and suppose that $|x| = |y|$. We let 
$\Delta_{xy}(\mu)$ denote the measure $\Delta(\mu)$, where $\Delta: \mathcal{U}^{x}\to \mathcal{U}^{x} \times \mathcal{U}^{y}$ via $a\mapsto (a,a)$.
\end{definition} 

\begin{example}\label{counter:exist} Let $M = (\mathbb{R};<)$ and $\mathcal{U}$ a monster model. We define the following types and measures:
\begin{enumerate} 
\item Let $p \in S_{x}(M)$ given by $\{x > a: a \in M\}$. 
\item Let $p_{+}$ be the unique global coheir of $p$ over $M$. More precisely, $p_{+}$ is finitely satisfiable in $M$ and given by $\{x > a: a \in M\} \cup \{ x < b: M < b\}$.  
\item Let $p_{+\infty}$ be the unique global heir of $p$. In particular, $p_{+\infty}$ is definable over $M$ and given by $\{x > a: a \in \mathcal{U}\}$. 
\item Let $\mu \in \mathfrak{M}_{x}(\mathcal{U})$ where $\mu := \frac{1}{2}\delta_{p_{+}} + \frac{1}{2}\delta_{p_{+\infty}}$. Notice that $\mu|_{M} = \delta_{p}$. Note that that $\mu$ is invariant over $M$, but neither definable or finitely satisfiable over $M$.
\item Let $\omega \in \mathfrak{M}_{xy}(\mathcal{U})$ where $\omega := \Delta_{xy}(\mu) =  \frac{1}{2}\Delta_{xy}(\delta_{p_{+\infty}}) + \frac{1}{2}\Delta_{xy}(\delta_{p_{+}})$. 
\begin{enumerate} 
\item Notice that $\pi_{x}(\omega) = \mu$. 
\item Notice that the formula $x = y$ is in both $\Delta_{xy}(\delta_{p_{+\infty}})$ and $\Delta_{xy}(\delta_{p_{+}})$. Hence $\omega( x=y ) = 1$. 
\item Notice that $\omega$ is not a separated amalgam. Fix $b \in \mathcal{U}\backslash M$ such that $M < b$. Consider the formula $\gamma(x,y) : =x > b \wedge y < b$. In particular, $\omega(\gamma(x,y)) = 0$ while $\pi_{x}(\omega)(x >b) \cdot  \pi_{y}(\omega)(y < b) = \frac{1}{2} \cdot \frac{1}{2} = \frac{1}{4}$. 
\end{enumerate} 
\item Consider  $\lambda := \omega|_{M}$. 
\begin{enumerate} 
\item Notice that $\pi_{x}(\lambda) = \mu|_{M}=\delta_p$ since $\pi_x(\omega)=\mu$.

\item By Lemma~\ref{type:amal}, $\lambda \in \mathfrak{M}_{xy}^{\am}(M)$, since $\pi_{x}(\lambda) = \delta_p(x)$.
\item In fact, $\lambda = \Delta_{xy}(\delta_p)$.
\item Since $\lambda = \omega|_{M}$, $\lambda( x = y) = 1$. 
\end{enumerate} 
\end{enumerate}
We now argue that $\amal(\lambda,\mu) = \emptyset$. Suppose not and let $\omega' \in \amal(\lambda,\mu)$. Fix $b \in \mathcal{U}$ such that $b > M$. Notice that $ x = y \subseteq [x \leq b \vee y \geq b]$.  Since $\omega'$ is a separated amalgam, 
\begin{equation*} 
1 = \lambda(x = y) = \omega'(x = y) \leq  \omega'(x \leq b \vee y \geq b) = 1 - \omega'( x > b \wedge y< b)
\end{equation*} 
\begin{equation*} 
= 1 - \omega'(x > b \wedge y = y)\omega'(y < b \wedge x = x) = 1 - \mu(x > b)\pi_{y}(\omega')(y < b)
\end{equation*} 
\begin{equation*} 
= 1 - \frac{1}{2}\pi_{y}(\omega')(y < b) \implies \pi_{y}(\omega')(y < b) = 0. 
\end{equation*} 
On the other hand, we have the following.
\begin{equation*} 
\frac{1}{2} = \mu( x < b) = \pi_{x}(\omega')(x < b) = \omega'(x < b \wedge y= y) = \omega'(x <b \wedge x = y) 
\end{equation*} 
\begin{equation*} 
= \omega' (y < b \wedge x = y) =\omega'(y<b \wedge x = x) = \pi_{y}(\omega')(y < b). 
\end{equation*} 
Hence we have a contradiction.
\end{example} 

With the previous example in mind, we now give the definition of amalgam deciphering.

\begin{definition}\label{def:amal-deci} Let $\mu \in \mathfrak{M}_{x}(\mathcal{U})$, $\nu \in \mathfrak{M}_{y}(\mathcal{U})$ and $A$ be a small subset of $\mathcal{U}$. We say that $\mu$ \textit{amalgam deciphers $\nu$ (over $A$)} if there exists $\lambda \in \mathfrak{M}_{xy}(A)$ such that 
\begin{enumerate}
\item $\pi_{x}(\lambda) = \mu|_{A}$. 
\item $\amal(\lambda,\mu) \neq \emptyset$. 
\item For every $\omega \in \amal(\lambda,\mu)$, $\pi_{y}(\omega) = \nu$. 
\end{enumerate}
We write $\mu\blacktriangleright_{A}\nu$ to mean ``$\mu$ amalgam deciphers $\nu$ over $A$" and $\mu\blacktriangleright\nu$ to mean ``there exists some small set $A$ such that $\mu  \blacktriangleright_{A} \nu$''.
\end{definition}
The following proposition essentially demonstrates all the pure theory we know about amalgam deciphering. 

\begin{proposition}\label{Basic:Amal} Let $A \subseteq B \subseteq \mathcal{U}$, $p \in S_{x}(\mathcal{U})$, $\mu \in \mathfrak{M}_{x}(\mathcal{U})$, $\nu \in \mathfrak{M}_{y}(\mathcal{U})$. 
\begin{enumerate} 
\item If $\mu \blacktriangleright_{A} \nu$, then $\mu \blacktriangleright_{B} \nu$. 
\item $\delta_{p} \geq_{\mathbb{E},A} \mu$ if and only if $\delta_{p} \blacktriangleright_{A} \mu$. 
\item If $\mu$ is A-invariant and $\mu \blacktriangleright_{A} \nu$, then $\nu$ is $A$-invariant. 
\item If $\mu$ is smooth and $\mu \blacktriangleright \nu$, then $\nu$ is smooth. 
\item If $\nu$ is smooth over $A$, then for any $\mu \in \mathfrak{M}_{x}(\mathcal{U})$, it follows that $\mu \blacktriangleright_{A} \nu$. 
\end{enumerate} 
\end{proposition} 

\begin{proof} Most of the statements above of slight variations of the proof as in the corresponding statement for extension domination and the only additional information needed is that the amalgamation space is non-empty. Therefore, the proofs are sketched. 
\begin{enumerate}
    \item Identical to Lemma \ref{lem:ext-para}. 
    \item Immediate from Lemma \ref{type:amal}. 
    \item Word for word as in the proofs of Lemma~\ref{aut:affine} and Proposition~\ref{pres:inv}.
    \item By Remark \ref{samalgam:exist}, fix $\omega \in \mathfrak{M}^{\am}_{xy}(\mathcal{U})$ such that $\pi_{x}(\omega) = \mu$ and $\pi_{y}(\omega) = \nu$. Consider $\lambda := \omega|_{A}$ where $\mu$ is smooth over $A$. Notice that $\amal(\lambda,\mu) \neq \emptyset$ since $\omega \in \amal(\lambda,\mu)$. The rest of the proof is word for word as in Proposition \ref{smooth:pres}.
    \item By Remark \ref{samalgam:exist}, fix $\omega \in \mathfrak{M}^{\am}_{xy}(\mathcal{U})$ such that $\pi_{x}(\omega) = \mu$ and $\pi_{y}(\omega) = \nu$. We claim $\lambda := \omega|_{A}$ witnesses $\mu \blacktriangleright_{A} \nu$. Notice that $\amal(\lambda,\mu) \neq \emptyset$ since $\omega \in \amal(\lambda,\mu)$. The rest of the proof is similar to Proposition \ref{smooth:all}. \qedhere
\end{enumerate}
\end{proof}

\subsection{Localizations and convex combinations}
While the general theory of amalgam deciphering remains mysterious, as previously advertised, amalgam deciphering interacts well with localizations and convex combinations. The proof of the following result is similar to Proposition \ref{local:ext}, so we only highlight the differences.

\begin{proposition}\label{amal:local} Let $\mu \in \mathfrak{M}_{x}(\mathcal{U})$, $\nu \in \mathfrak{M}_{y}(\mathcal{U})$, and $\varphi(x,b) \in \mathcal{L}_{y}(b)$. Suppose that  $0< \nu(\varphi(y,b)) < 1$.  If $\mu \blacktriangleright_{A} \nu$, then $\mu \blacktriangleright_{Ab} \nu_{[\varphi]}$. 
\end{proposition}

\begin{proof} Let $r := \nu(\varphi(y,b))$ and $s = \nu(\neg \varphi(y,b))$. By Proposition \ref{Basic:Amal}, 
choose $\lambda \in \mathfrak{M}^{\am}_{xy}(Ab)$ such that $\lambda$ witnesses $\mu \blacktriangleright_{Ab} \nu$. Let $\Phi(x,y) := \varphi(y,b) \wedge x = x$ and  consider the measures $\lambda_{[\Phi]}$, $\lambda_{[\neg \Phi]}\in \mathfrak{M}_{xy}(Ab)$. We argue that $\lambda_{[\Phi]}$ witness $\mu \blacktriangleright_{Ab} \nu$.
\begin{enumerate} 
 \item Claim: $\pi_{x}(\lambda_{[\Phi]}) = \mu|_{Ab}$. The proof is the same as Proposition~\ref{local:ext}.
\item Claim: $\lambda_{[\Phi]} \in \mathfrak{M}_{xy}^{\am}(Ab)$. Let $\psi(x) \in \mathcal{L}_{x}(Ab)$ and $\theta(y) \in \mathcal{L}_{y}(Ab)$. Then 
 \begin{equation*}
     \lambda_{[\Phi]}(\psi(x) \wedge \theta(y)) = \frac{\lambda(\psi(x) \wedge \theta(y) \wedge \varphi(y,b))}{\lambda(\varphi(y,b) \wedge x = x)} 
 \end{equation*}
 \begin{equation*} 
 \overset{(*)}{=} \frac{\mu(\varphi(x))\nu(\theta(y) \wedge \varphi(y,b))}{\nu(\varphi(y,b))} = \mu(\varphi(x)) \cdot \nu_{[\varphi]}(\theta(y)),
\end{equation*} 
where equality $(*)$ holds since $\pi_{y}(\lambda) = \nu|_{Ab}$ and $\lambda$ is a separated amalgam. By Lemma \ref{Amal:eq}, the claim holds. 
\item Claim: $\amal(\lambda_{[\Phi]},\mu) \neq \emptyset$ and $\amal(\lambda_{[\neg \Phi]},\mu) \neq \emptyset$.  By assumption, $\amal(\lambda,\mu) \neq \emptyset$. Let $\omega \in \amal(\lambda,\mu)$ and consider $\omega_{[\Phi]}$ and $\omega_{[\neg \Phi]}$. A straightforward computation show that $\omega_{[\Phi]} \in \amal(\lambda_{[\Phi]},\mu)$ and $\omega_{[\neg \Phi]} \in \amal(\lambda_{[\neg \Phi]},\mu)$. 
\item Fix $\omega \in \amal(\lambda_{[\Phi]},\mu)$. It suffices to show that $\pi_{y}(\omega) = \nu_{[\varphi]}$. 
\item Let $\eta \in \amal(\lambda_{[\neg \Phi]},\mu)$. Consider $\omega' = r\omega + s\eta$. A straightforward computation show that $\omega'|_{A} = \lambda$ and $\pi_{x}(\omega') = \mu$.
\item Claim: $\omega' \in \mathfrak{M}^{\am}_{xy}(\mathcal{U})$. Let $\psi(x) \in \mathcal{L}_{x}(\mathcal{U})$ and $\theta(y) \in \mathcal{L}_{y}(\mathcal{U})$. Then a straightforward computation shows
\begin{equation*}
    \omega'(\psi(x) \wedge \theta(y)) 
=\mu(\psi(x))(r\pi_{y}(\omega) + s\pi_{y}(\eta))(\theta(y)). 
\end{equation*}
So the claim holds by Lemma \ref{Amal:eq}. 
\item Hence $\omega' \in \amal(\lambda,\mu)$. By amalgam deciphering, $\pi_{y}(\omega') = \nu$. A similar argument to Proposition \ref{local:ext}(6), shows $\pi_{y}(\omega) = \nu_{[\varphi]}$, completing the proof. \qedhere
\end{enumerate} 
\end{proof}

Our next goal is to show that amalgam deciphering is well-behaved with respect to a certain kind of convex combinations. We begin with a lemma and advise the reader to recall Definition \ref{support}. 

\begin{lemma}\label{lemma:cut} Suppose that $\nu_1 \in \mathfrak{M}_{y}(\mathcal{U})$ and $\nu_2 \in \mathfrak{M}_{y}(\mathcal{U})$ and $\supp(\nu_1) \cap \supp(\nu_2) = \emptyset$. Then there exists a formula $\gamma(y) \in \mathcal{L}_{y}(\mathcal{U})$ such that $\supp(\nu_1) \subseteq [\gamma(y)]$ and $\supp(\nu_2) \cap [\gamma(y)] = \emptyset$. 
\end{lemma} 

\begin{proof} Since $\supp(\nu_2)$ is closed, we have that $S_{y}(\mathcal{U}) \backslash \supp(\nu_2)$ is an open subset of $S_{y}(\mathcal{U})$. Then $S_{y}(\mathcal{U}) \backslash \supp(\nu_2) = \bigcup_{i \in I} [\gamma_i(y)]$ and since $\supp(\nu_1) \subseteq S_{y}(\mathcal{U}) \backslash \supp(\nu_2)$, this collection of basic clopens is an open cover of $\supp(\nu_1)$. We have that $\supp(\nu_1)$ is compact and so there is a finite subcover, say $\bigcup_{i=1}^{n} [\gamma_i(y)]$. Setting $\gamma : = \bigvee_{i=1}^n \gamma_i(y)$ concludes the lemma. 
\end{proof} 

\begin{proposition}\label{prop:conv} Suppose that $\mu \blacktriangleright_{A} \nu_1$ and $\mu \blacktriangleright_{A} \nu_2$ such that $\supp(\nu_1) \cap \supp(\nu_2) = \emptyset$. Then there exists a finite set $b$ such that for any $r, s \in [0,1]$ with $r+s =1$, we have that $\mu \blacktriangleright_{Ab} r\nu_1 + s\nu_2$. 
\end{proposition}

\begin{proof}
Assume the conditions above. By Lemma \ref{lemma:cut}, there exists $\varphi(y,b) \in \mathcal{L}_{y}(\mathcal{U})$ such that $\supp(\nu_1) \subseteq [\varphi(x,b)]$ and $[\varphi(x,b)] \cap \supp(\nu_2) = \emptyset$. By Proposition \ref{Basic:Amal}, $ \mu \blacktriangleright_{Ab} \nu_i$.  Let  $\lambda_i \in \mathfrak{M}_{xy}^{\am}(Ab)$ witness $\mu \blacktriangleright_{Ab} \nu_i$ for $i \in \{1,2\}$. Consider the measure $\lambda := r\lambda_1 + s\lambda_2$. We claim that $\lambda$ witnesses $\mu \geq_{Ab} r\nu_1 + s\nu_2$.

\begin{enumerate} 
\item Claim: $\pi_{x}(\lambda) = \mu|_{Ab}$. Notice that 
\begin{equation*} 
\pi_{x}(\lambda) = \pi_{x}(r\lambda_1 + s\lambda_2) = r\pi_{x}(\lambda_1)+s\pi_{x}(\lambda_2) = r\mu|_{Ab} + s\mu|_{Ab} = \mu|_{Ab}.
\end{equation*} 
\item Claim: $\amal(\lambda,\mu) \neq \emptyset$. By assumption, $\amal(\lambda_1,\mu), \amal(\lambda_2,\mu)\neq \emptyset$. Take $\omega_i \in \amal(\lambda_i,\mu)$. It follows immediately $r\omega_1 + s\omega_2 \in \amal(\lambda,\mu)$. 
\item Fix $\omega \in \amal(\lambda,\mu)$. It suffices to show that $\pi_{y}(\omega) = \nu$. Consider the formula $\Phi(x,y) := \varphi(y,b) \wedge x = x$. 
\item Claim: $\omega_{[\Phi]} \in \amal(\lambda_{1},\mu)$. 
\begin{enumerate}[($i$)] 
\item Subclaim: $\pi_{x}(\omega_{[\Phi]}) = \mu$. Let $\psi(x) \in \mathcal{L}_{x}(\mathcal{U})$ and notice 
\begin{equation*} \pi_{x}(\omega_{[\Phi]})(\psi(x)) =\omega_{[\Phi]}(\psi(x) \wedge y=y) =  \frac{\omega(\psi(x) \wedge \varphi(y,b))}{\omega(\varphi(y,b) \wedge x = x)} 
\end{equation*}
\begin{equation*}
    \overset{(*)}{=} \frac{\pi_{x}(\omega)(\psi(x)) \cdot \pi_{y}(\omega)(\varphi(y,b))}{\pi_{y}(\omega)(\varphi(y,b))} = \pi_{x}(\omega)(\psi(x)) = \mu(\psi(x)). 
\end{equation*} 
The equation $(*)$ follows from the fact that $\omega \in \mathfrak{M}_{x}^{\am}(\mathcal{U})$. 
\item Subclaim: $\omega_{[\Phi]}|_{Ab} = \lambda_1$. Fix $\theta(x,y) \in \mathcal{L}_{xy}(Ab)$. Now consider 
\begin{equation*} 
\omega_{[\Phi]}(\theta(x,y)) =   \frac{\omega(\theta(x,y) \wedge \varphi(y,b))}{\omega(\varphi(y,b) \wedge x = x)} 
\end{equation*} 
\begin{equation*}
\overset{(a)}{=} \frac{(r\lambda_1 + s\lambda_2)(\theta(x,y) \wedge \varphi(y,b))}{(r\lambda_1 + s\lambda_2)(\varphi(y,b) \wedge x= x)}
\end{equation*} 

\begin{equation*} 
=\frac{r\lambda_1(\theta(x,y) \wedge \varphi(y,b)) + s\lambda_2(\theta(x,y) \wedge \varphi(y,b))}{r\lambda_1(\ \varphi(y,b) \wedge x = x) + s\lambda_2(\ \varphi(y,b) \wedge x = x)}
\end{equation*} 
\begin{equation*} 
\overset{(b)}{=} \frac{r\lambda_1(\theta(x,y) \wedge \varphi(y,b))}{r\lambda_1(\varphi(y,b)\wedge x = x)} \overset{(c)}{=} \lambda_1(\theta(x,y)). 
\end{equation*} 
Equation $(a)$ is justified since $\theta(x,y) \wedge \varphi(y,b) \in \mathcal{L}_{y}(Ab)$ and $\omega|_{Ab} = \lambda$. Equation $(b)$ is justified since $\lambda_{2}(\varphi(y,b) \wedge x = x) = \nu_2(\varphi(y,b)) = 0$. Equation $(c)$ is justified since $\lambda_1(\varphi(y,b)\wedge x = x) = \nu(\varphi(y,b)) = 1$. 
\item Subclaim: $\omega_{[\Phi]} \in \mathfrak{M}_{xy}^{\amal}(\mathcal{U})$. Let $\chi(x) \in \mathcal{L}_{x}(\mathcal{U})$ and $\psi(y)) \in \mathcal{L}_{x}(\mathcal{U})$. Then 
\begin{equation*}
    \omega_{[\Phi]}(\chi(x) \wedge \psi(y)) = \frac{\omega(\chi(x) \wedge \psi(y) \wedge \varphi(y,b))}{\omega(\varphi(y,b) \wedge x=x)}
\end{equation*}
\begin{equation*}
    = \pi_{x}(\omega)(\chi(x)) \cdot \frac{\pi_{y}(\omega)(\psi(y)\wedge \varphi(y,b))}{\pi_{y}(\omega)(\varphi(y,b))}  
\end{equation*}
\begin{equation*}
    = \pi_{x}(\omega)(\chi(x)) \cdot (\pi_{y}(\omega))_{[\varphi]}(\psi(y))
\end{equation*}
By Lemma \ref{Amal:eq}, the subclaim holds. 
\end{enumerate} 
\item Claim: $\omega_{[\neg \Phi]} \in \amal(\lambda_2, \mu)$. The proof is similar to (4). 
\item By amalgam deciphering, $\pi_{y}(\omega_{[\Phi]}) = \nu_1$ and $\pi_{y}(\omega_{[\neg\Phi]}) = \nu_2$. 
\item Claim: $\pi_{y}(\omega) = r\pi_{y}(\omega_{[\Phi]}) + s\pi_{y}(\omega_{[\neg \Phi]})$. Fix $\rho(y) \in \mathcal{L}_{y}(\mathcal{U})$. Notice 
\begin{equation*}
\pi_{y}(\omega)(\rho(y)) = \omega(\rho(y) \wedge x = x) 
\end{equation*} 
\begin{equation*} 
= \omega(\rho(y) \wedge \varphi(y,b) \wedge x = x) + \omega(\rho(y) \wedge \neg \varphi(y,b) \wedge x = x)
\end{equation*} 
\begin{equation*} 
= r\omega_{[\Phi]}(\rho(y) \wedge x = x) + s\omega_{[\neg \Phi]}(\rho(y) \wedge x = x). 
\end{equation*} 
\begin{equation*} 
= r\pi_{y}(\omega_{[\Phi]})(\rho(y)) + s\pi_{y}(\omega_{[\neg \Phi]})(\rho(y)). 
\end{equation*} 
\item We finally show that $\pi_{y}(\omega) = r\nu_1 + s\nu_2$. Notice 
\begin{equation*}
    \pi_{y}(\omega) =  r\pi_{y}(\omega_{[\Phi]}) + s\pi_{y}(\omega_{[\neg\Phi]}) = r\nu_1 + s \nu_2. 
\end{equation*}
where the second equation follows from $(6)$. \qedhere 
\end{enumerate} 
\end{proof}

We can derive the following results about extension domination from the previous two results. 

\begin{corollary}\label{Ex:type} Let $r,s \in [0,1]$ such that $r + s = 1$. Let $p \in S_{x}(\mathcal{U})$, and $\nu_1,\nu_2 \in \mathfrak{M}_{y}(\mathcal{U})$ such that $\supp(\nu_1) \cap \supp(\nu_2) = \emptyset$. Suppose that $\delta_{p} \geq_{\mathbb{E},A} \nu_1$ and $\delta_{p} \geq_{\mathbb{E},A} \nu_2$. Then $\delta_{p} \geq_{\mathbb{E},Ab} r\nu_1 + s\nu_2$ for some finite set $b$.  

In particular, if $p\geq_{D,A} q_i$ for $i=1,\ldots,n$, then there is a finite set $b$ such that $\delta_p\geq_{\mathbb{E},Ab} \sum_{i=1}^{n} r_i\delta_{q_i}$ for $r_i\geq 0$ and $\sum_{i=1}^{n} r_i=1$.

\end{corollary}

\begin{proof} 
By Lemma~\ref{lemma:cut} there exists $\varphi(y,b)$ such that $\nu_1(\varphi(y,b)) = 1$ and $\nu_2(\varphi(y,b)) = 0$. By Lemma~\ref{lem:ext-para}, $\delta_{p} \geq_{\mathbb{E},Ab} \nu_1$ and $\delta_{p} \geq_{\mathbb{E},Ab} \nu_2$. By Lemma~~\ref{Basic:Amal}, $\delta_{p} \blacktriangleright_{Ab} \nu_1$ and $\delta_{p} \blacktriangleright_{Ab} \nu_2$. By Proposition \ref{prop:conv}, $\delta_{p} \blacktriangleright_{Ab} r\nu_1 + s\nu_2$. Again, by Lemma~~\ref{Basic:Amal}, $\delta_{p} \geq_{\mathbb{E},Ab} r\nu_1 + s\nu_2$.

The \textit{in particular} portion follows immediately from the previous paragraph and Proposition \ref{dom:exdom}. 
\end{proof}

\section{Examples} 

We construct some explicit examples of extension domination. These examples demonstrate that extension domination is a non-trivial relation. More importantly, these computations demonstrate how to effectively show when one measure extension dominates another. Our first examples occur in Presburger arithmetic. There, we use a restriction of the Morley product to construct a witness for extension domination. Later, we will work in the Fr\"{a}ss\'{e} limit of ordered equivalence relations. For the measures we consider here, the restriction of the Morley product does not witness extension domination. However, a witness does exist and is constructed by considering \textit{smooth extensions}. 

We begin with a lemma which make our calculations easier. The proof is a straightforward computation and left to the readers.

\begin{lemma}\label{intersection} Let $\mu,\nu \in \mathfrak{M}_{x}(\mathcal{U})$. If $\mu$ and $\nu$ agree on a basis i.e. a collection of definable sets (in variable $x$) $\mathcal{B}$ closed under finite intersections such that any definable subset $X$ (in variable $x$) is a finite union of sets in $\mathcal{B}$, then $\mu = \nu$. 
\end{lemma}

We also recall the definition of the \textit{Morley product} for invariant measures. This is well-defined in the NIP setting since invariance is equivalent to Borel-definable.  
\begin{definition} Suppose that $T$ is an NIP theory, $M \models T$, $\mathcal{U}$ is a monster model of $T$, $\mu \in \mathfrak{M}_{x}(\mathcal{U})$, $\nu \in \mathfrak{M}_{y}(\mathcal{U})$, and $\mu$ is invariant over a small submodel $M$. Then we can define a measure $\mu \otimes \nu$ in $\mathfrak{M}_{xy}$, known as the Morley product, where for any formula $\varphi(x,y;z) \in \mathcal{L}_{xy,z}(\emptyset)$, we have that 
\begin{equation*}
    \mu \otimes \nu(\varphi(x,y,b)) = \int_{S_{x}(N)} F_{\mu,N}^{\varphi}d(\widehat{\nu|_{N}}). 
\end{equation*}
where
\begin{enumerate}
    \item $N$ is a small model containing $Mb$. 
    \item $F_{\mu,N}^{\varphi}(q) = \mu(\varphi(x,c,b))$ where $c \in \mathcal{U}$ and $c \models q$.
    \item $\widehat{\nu|_{N}}$ is the unique regular Borel probability measure on $S_{y}(N)$ which extends $\nu|_{N}$ on definable sets (See Fact~\ref{Basic:Keisler}(3)). We will simply write $\widehat{\nu|_{N}}$ simply as $\nu|_{N}$. 
\end{enumerate}
\end{definition}



\subsection{Presburger Arithmetic} 
Let $M = (\mathbb{Z};+,<,0)$ and $\mathcal{U}$ be a monster model of $\mathcal{U}$. We recall that this theory is NIP. We give some examples of extension domination in this structure. 

\begin{proposition}\label{Z:ex} Let $\mu \in \mathfrak{M}_{x}(\mathcal{U})$ and $\nu \in \mathfrak{M}_{y}(\mathcal{U})$.
\begin{enumerate}
    \item Suppose that for every $b \in \mathcal{U}$, $\mu(x > b) = \nu(x > b) = 1$. Then $\mu \geq_{\mathbb{E},\mathbb{Z}} \nu$.
    \item Suppose that for every $b \in \mathcal{U}$, $\mu(x < b) = \nu(x < b) = 1$. Then $\mu \geq_{\mathbb{E},\mathbb{Z}} \nu$.
\end{enumerate}
\end{proposition} 

\begin{proof} The proofs are similar and so we only prove $(1)$. It is clear that for each $p \in \supp(\nu)$, $p$ is $\mathbb{Z}$-invariant (even $\mathbb{Z}$-definable). It follows from~\cite[Lemma 2.10]{chernikov2022definable} that $\nu$ is $\mathbb{Z}$-invariant. Thus the Morley product $\nu\otimes \mu$ is well-defined. Hence, consider the measure $\lambda: = \nu_{y} \otimes \mu_{x}|_{\mathbb{Z}}$. Let $\varphi(x,y) := y > x$ and notice that

\begin{equation*}
    \lambda(y > x) = \nu_{y} \otimes \mu_{x}( y > x) = \int_{S_{x}(\mathbb{Z})} F_{\nu_{y},\mathbb{Z}}^{\varphi} d\mu_{x} = \int_{S_{x}(\mathbb{Z})} \mathbf{1} d\mu = 1.
\end{equation*}
The second equality follows since $F_{\nu_{y},\mathbb{Z}}^{\varphi}(q) = \nu(y > b) = 1$ where $b \models q$.

Let $\omega \in \E(\lambda,\mu)$. We now show that for every $a \in \mathcal{U}$, $\pi_{y}(\omega)(a<y) = 1$. Fix $a \in \mathcal{U}$ and consider the following computation: 
\begin{equation*}
    \pi_{y}(\omega)(a < y)) = \omega(a < y \wedge x = x) \overset{(*)}{=} \omega(a <y \wedge x < y) 
\end{equation*}
\begin{equation*}
= \omega((a < x \wedge x < y) \vee (x \leq a <  y)) \geq \omega( a < x \wedge x < y)
\end{equation*}
\begin{equation*}
     \overset{(*)}{=} \omega( a < x \wedge y = y) = \pi_{x}(\omega)( a < x) = 1. 
\end{equation*}
The $(*)$ equations are justified by the fact that $\omega( y > x) = \lambda( y > x) = 1$. 

We now argue that for every $\omega \in \E(\lambda,\mu)$, $\pi_{y}(\omega) = \nu$. For each $n \geq 2$, let $D_{n}(x) := \exists z( n\cdot z = x)$. Every definable subset of $\mathcal{U}$ is a Boolean combination of intervals and formulas of the form $D_{n}(x+k)$ where $k \leq n$. We let $I(y)$ be of the form $x <b, b < x, x = b, x \neq b$ and consider the set $\mathcal{B}$ where  
\begin{align*}
    \mathcal{B} = \Bigg\{ & \bigwedge_{j = 1}^{n} I_{j}^{\varepsilon_{1}(j)}(y) \wedge \bigwedge_{l = 1}^{m} D_{s_{l}}^{\varepsilon_{2}(l)}(y + k_l): n,m,s_{l} \in \mathbb{N}, k_l \leq s_l,
    \\ & \varepsilon_{1}: \{1,...,n\} \to \{0,1\}, \varepsilon_{2}:\{1,...,m\} \to \{0,1\} \Bigg\}. 
\end{align*}
Any definable subset of $\mathcal{U}^{y}$ can be written as a finite union of elements from $\mathcal{B}$. We now show that $\pi_{y}(\omega) = \nu$. By Lemma \ref{intersection}, it suffices to show that for any $\theta(y) \in \mathcal{B}$, $\pi_{y}(\omega)(\theta(y)) = \nu(\theta(y))$. Let $\theta(y) = \bigwedge_{j = 1}^{n} I_{j}^{\varepsilon_{1}(j)}(y) \wedge \bigwedge_{l = 1}^{m} D_{s_{l}}^{\varepsilon_{2}(l)}(y + k_l)$. Let $c \in \mathcal{U}^{y}$ such that $c$ is larger than all the parameters which occur in $\theta(y)$. Then either: 
\begin{enumerate}
    \item $[\bigwedge_{j = 1}^{n} I_{j}^{\varepsilon_{1}(j)}(y) \wedge y > c] = \emptyset$ and so $[\theta(y) \wedge y > c] = \emptyset$. In this case, we have that 
    \begin{equation*}
        \pi_{y}(\omega)(\theta(y)) \overset{(*)}{=} \pi_{y}(\omega)(\theta(y) \wedge y > c) = \pi_{y}(\omega)(\emptyset ) = 0 
    \end{equation*}
    \begin{equation*}
        = \nu(\emptyset) = \nu(\theta(y) \wedge y > c) \overset{(*)}{=} \nu(\theta(y)). 
    \end{equation*}
The equality $(*)$ follows from the fact that $\omega(y > c) =\nu(y>c) =  1$. 
    \item Or $[\bigwedge_{j = 1}^{n} I_{j}^{\varepsilon_{1}(j)}(y) \wedge y > c] = [y > c]$ and so
 \begin{equation*}
        \pi_{y}(\omega)(\theta(y)) \overset{(*)}{=} \pi_{y}(\omega)(\theta(y) \wedge y > c) = \pi_{y}(\omega)\left(\bigwedge_{l=1}^{m} D_{s_l}^{\varepsilon_{2}(l)}(y+k_l) \wedge y >c \right)
    \end{equation*}
    \begin{equation*}
        \overset{(*)}{=} \pi_{y}(\omega)\left(\bigwedge_{l=1}^{m} D_{s_l}^{\varepsilon_{2}(l)}(y+k_l) \right) \overset{(\dagger)}{=} \pi_{y}(\lambda)\left( \bigwedge_{l=1}^{m} D_{s_l}^{\varepsilon_{2}(l)}(y+k_l) \right)
    \end{equation*}
    \begin{equation*}
        =\nu\left( \bigwedge_{l=1}^{m} D_{s_l}^{\varepsilon_{2}(l)}(y+k_l) \right) \overset{(*)}{=} \nu\left( \bigwedge_{l=1}^{m} D_{s_l}^{\varepsilon_{2}(l)}(y+k_l) \wedge y > c \right) = \nu(\theta(y)). 
    \end{equation*}
    Again, equality $(*)$ follows from the fact that $\omega(y > c) =\nu(y>c) =  1$. Equality $(\dagger)$ follows from the fact that the formula $\bigwedge_{l=1}^{m} D^{\varepsilon_{2}(l)}_{s_l}(y+k_l) \in \mathcal{L}_{y}(\mathbb{Z})$. 
\end{enumerate}
Hence, we conclude that $\pi_{y}(\omega) = \nu$ and so $\mu \geq_{\mathbb{E},\mathbb{Z}} \nu$. 
\end{proof}

\begin{corollary}\label{Z:ex2} Suppose that for every $b \in \mathcal{U}$, $\mu (x > b) = \nu(x < b) = 1$. Then \begin{enumerate}
    \item $\mu \geq_{\mathbb{E},\mathbb{Z}} \nu$.
    \item $\nu \geq_{\mathbb{E},\mathbb{Z}} \mu$. 
\end{enumerate}
\end{corollary}

\begin{proof} Again, the proofs are similar and so we only prove $(1)$. Consider the definable function $f:\mathcal{U}^{x} \to \mathcal{U}^{y}$ via $f(a) = -a$. Then $\mu \geq_{\mathbb{E},\mathbb{Z}} f(\mu)$ by Proposition~\ref{pushforward} and $f(\mu)(x<b) = 1$ for every $b \in \mathcal{U}$. By Proposition \ref{Z:ex}, we conclude that $f(\mu) \geq_{\mathbb{E},\mathbb{Z}} \nu$. By transitivity of extension domination, $\mu \geq_{\mathbb{E},\mathbb{Z}} \nu$. 
\end{proof}

\begin{corollary} Suppose that for every $b \in \mathcal{U}$, $\mu(x > b) = \eta(x > b) = \nu(x < b) =1$ and $r,s \in [0,1]$ such that $r + s =1$. Then $\mu \geq_{\mathbb{E},\mathbb{Z}} r\nu + s\eta$. 
\end{corollary}

\begin{proof} Let $p \in S_{x}(\mathcal{U})$ such that $x > b \in p$ for every $b \in \mathcal{U}$. 
By Proposition \ref{Z:ex}, $\mu \geq_{\mathbb{E},\mathbb{Z}} \delta_{p}$ and $\delta_{p} \geq_{\mathbb{E},\mathbb{Z}} \eta$. 
By Corollary \ref{Z:ex2}, $\delta_{p} \geq_{\mathbb{E},\mathbb{Z}} \nu$. 
Since $\supp(\eta) \cap \supp(\nu) = \emptyset$,  $\supp(\eta) \subset [x>0]$ and $\supp(\nu) \cap [x > 0] = \emptyset$, by Corollary \ref{Ex:type} $\delta_{p} \geq_{\mathbb{E},\mathbb{Z}} r\nu +s\eta$. By transitivity, we conclude that $\mu \geq_{\mathbb{E},\mathbb{Z}} r\nu + s\eta$.
\end{proof}

\subsection{Ordered equivalence relation}

The following example is slightly more complicated. We construct two global measures $\mu$ and $\nu$ which are generically stable and dominate one another.

\begin{definition} Suppose that $T$ is NIP and $\mu \in \mathfrak{M}_{x}(\mathcal{U})$ for some monster model $\mathcal{U}$. We say that $\mu$ is \textit{generically stable} if $\mu$ is definable and finitely satisfiable over some small model $M$. 
\end{definition}

\begin{lemma}\label{distal} Let $T$ be an NIP theory in the language $\mathcal{L}$ and let $T'$ be the reduct of $T$ to the language $\mathcal{L}'$ such that $T'$ is distal. Let $\mathcal{U} \models T$, $M$ is a small submodel of $\mathcal{U}$, and $\mu \in \mathfrak{M}_{x}(\mathcal{U})$ such that $\mu$ is generically stable over $M$.  Let $\mathcal{U}'$ and $M'$ be the associated models of the reduct. Then $\mu'=\mu|_{\mathcal{L}'(\mathcal{U'})}$ is smooth over $M'$. In particular, for any $\theta(x) \in \mathcal{L'}_{x}(\mathcal{U})$ and $\nu \in \mathfrak{M}_{x}(\mathcal{U})$, if $\nu|_M = \mu|_M$ then $\nu(\theta(x)) = \mu(\theta(x))$. 
\end{lemma}

\begin{proof} We first claim that $\mu'$ is generically stable over $M'$. It suffices to show that $\mu'$ is both definable and finitely satisfiable over $M'$. 

\begin{enumerate}
    \item Claim 1: $\mu'$ is finitely satifiable over $M'$ follows directly by definition. 
    \item Claim 2: $\mu'$ is definable over $M'$. Note that by claim 1, $\mu'$ is invariant over $M'$. Hence for any formula $\varphi(x,z) \in \mathcal{L'}_{xz}(\emptyset)$, the map $F_{\mu',M'}^{\varphi'}: S_{z}(M') \to [0,1]$ is well-defined. Let $s:S_{z}(M) \to S_{z}(M')$ be the obvious restriction map. This map is continuous. By \cite[Proposition 2.9]{conant2020remarks}, for any closed $C\subseteq [0,1]$,  $(F_{\mu,M}^{\varphi})^{-1}(C)$ is $\varphi^{*}$-type definable. Hence for any $q_1,q_2 \in S_{z}(M)$ and $b_1 \models q_1$, $b_2 \models q_2$, if $\tp_{\varphi^{*}}(b_1/M) = \tp_{\varphi^{*}}(b_2/M)$, then $\mu(\varphi(x,b_1)) = \mu(\varphi(x,b_2))$. From this observation, one can check that $s((F^{\varphi}_{\mu,M})^{-1}(C)) = (F^{\varphi}_{\mu',M'})^{-1}(C)$. Notice that $s((F^{\varphi}_{\mu,M})^{-1}(C))$ is closed since it is the image of a compact set (which is closed). Therefore, $(F_{\mu',M'}^{\varphi})^{-1}(C)$ is closed and hence $F_{\mu',M'}^{\varphi}$ is continuous. 
\end{enumerate}
Hence $\mu'$ is generically stable over $M'$. Since $T'$ is distal, all generically stable measures are smooth \cite[Theorem 1.1]{simon2013distal} and so $\mu'$ is smooth over $M'$ \cite[Lemma 7.17]{simon2015guide}. The \textit{in particular} portion follows immediately since $\nu'|_{M'} = \mu'|_{M'}$ and $\mu'$ is smooth over $M'$. 
\end{proof}

\begin{example}\label{OER} Let $\mathcal{L} = \{<,E\}$ be the language of an \textit{ordered equivalence relation}. Then finite linear orders with an equivalence relation form a Fra\"{i}ss\'{e} class and we let $T_{OE}$ be the theory of the Fra\"{i}ss\'{e} limit. Let $\mathcal{U}$ be a monster model of $T_{OE}$ and $M$ a small elementary submodel. Notice the following:  
\begin{enumerate} 
\item $T_{OE}$ admits quantifier elimination. 
\item $T_{OE}$ restricted to the language $\{<\}$ is the theory of a dense linear order without endpoints. 
\item $T_{OE}$ restricted to the language $\{E\}$ is the theory of infinitely many equivalence classes all with infinitely many elements. 
\item For any $a \in \mathcal{U}$, the $E(x,a)$ is a dense-codense subset of $\mathcal{U}$. 
\item $T_{OE}$ is NIP. 
\end{enumerate} 
Let $p$ be the unique type in $S_{x}(\mathcal{U})$ extending the collection of formulas $\{\neg E(a,x): a \in \mathcal{U}\} \cup \{x > a: a \in \mathcal{U}\}$. Note that $p$ is definable over any model by quantifier elimination and so $p$ is definable over $M$. Let $I = (a_i)_{i \in [0,\infty)}$ be a Morley sequence in $p$ over $M$. By construction, we notice that: 

\begin{enumerate} 
\item $\mathcal{U} \models a_i < a_j$ if and only if $i < j$. 
\item $\mathcal{U} \models  E(a_i,a_j)$ if and only if $i= j$. 
\end{enumerate} 
We now construct generically stable (non-smooth) measures $\mu$ and $\nu$ in  $ \mathfrak{M}_{x}(\mathcal{U})$ such that $\mu \geq_{\mathbb{E}} \nu$ and $\nu \geq_{\mathbb{E}} \mu$.  Let $L$ denote the usual Lebesgue measure on $\mathbb{R}$. 
\begin{enumerate} 
\item For any $\varphi(x) \in \mathcal{L}_{x}(\mathcal{U})$, we define $\mu(\varphi(x)) = L(\{i \in [0,1]: \mathcal{U} \models \varphi(a_i)\})$. 
\item For any $\varphi(x) \in \mathcal{L}_{x}(\mathcal{U})$, we define $\nu(\varphi(x)) = L(\{i \in [2,3]: \mathcal{U} \models \varphi(a_i)\})$.
\end{enumerate} 
Since $T_{OE}$ is NIP and our sequence $(a_i)_{i \in [0,\infty)}$ is indiscernible, $\mu$ and $\nu$ are well-defined and generically stable (see~\cite[Example 7.32]{simon2015guide}). More explicitly, if we fix a small model $N$ such that $M \cup \{a_i: i \in [0,\infty)\} \subseteq N$ and $M \prec N \prec \mathcal{U}$, then both $\mu$ and $\nu$ are generically stable over $N$. We now argue that $\mu$ and $\nu$ are not smooth over any subset of $\mathcal{U}$. Let $\mathcal{U} \prec \mathcal{U}'$ and consider $b \in \mathcal{U}'$ such that $b \models p$.
\begin{enumerate}
\item Since $\mu$ is generically stable, $\mu$ admits a definable extension to $\mathcal{U}'$, namely $\mu'$. The measure $\mu'$ is still definable over $N$ and hence $\mu(E(x,b)) = 0$. 
\item Let $\varphi(x) = E(x,b)$. For any $\psi(x) \in \mathcal{L}_{x}(\mathcal{U})$ such that $\mathcal{U}' \models \varphi(x) \to \psi(x)$, it follows that there exists a finite (possibly empty) set $A$ such that $A \subseteq \mathcal{U}$ and $\psi(x) \equiv \bigwedge_{a \in A} x \neq a$. Hence, if follows that  $\inf\{\mu(\psi(x)): \psi(x) \in \mathcal{L}_{x}(\mathcal{U})\,\text{ and } \mathcal{U}' \models \varphi(x) \to \psi(x)\} = 1$. 

By construction, $\mu(\varphi(x)) = 1$. Likewise for any $\varphi(x) \in \mathcal{L}_{x}(\mathcal{U})$ such that $\mathcal{U}' \models \psi(x) \to \varphi(x)$, we have that $\varphi(x) \equiv x \neq x$ and so $\mu(\varphi(x)) = 0$. Therefore, for any $r \in [0,1]$, there exists a measure $\mu_{r} \in \mathfrak{M}_{x}(\mathcal{U}')$ such that $\mu_{r}|_{\mathcal{U}} = \mu$ and $\mu_{r}(E(x,b)) = r$ by Fact~\ref{fact:charges-r}. 
\end{enumerate} 
We claim that $\mu \geq_{\mathbb{E},N} \nu$ (and $\nu \geq_{\mathbb{E},N} \mu$). Choose $b \in \mathcal{U} \backslash N$. Then there exists unique global measures $\mu_{b}$ and $\nu_{b}$ such that 
\begin{enumerate}
    \item $\mu_{b}$ and $\nu_{b}$ are smooth over $Nb$.
    \item $\mu_{b}|_{\mathcal{L}_{<}} = \mu|_{\mathcal{L}_{<}}$ and $\mu(E(x,b)) = 1$. 
    \item $\nu_{b}|_{\mathcal{L}_{<}} = \nu|_{\mathcal{L}_{<}}$ and $\nu(E(x,b)) = 1$. 
    \item $\mu_{b}|_{N} = \mu|_{N}$ and $\nu_{b}|_{N} = \nu|_{N}$
\end{enumerate}
The measures $\mu_{b}$ and $\nu_{b}$ are \textit{smooth extensions} of $\mu|_{N}$ and $\nu|_{N}$ respectively. Since $\mu_{b}$ and $\nu_{b}$ are smooth, there exists a unique measure $\eta \in \mathfrak{M}_{xy}(\mathcal{U})$ such that $\eta$ is a separated amalgam of $\mu_{b}$ and $\nu_{b}$ by~\cite[Corollary 2.5]{NIP-III}. Moreover, $\eta = \mu_{b}(x) \otimes \nu_{b}(y)$ (by uniqueness). We let $\lambda := \eta|_{N}$. We show that $\lambda(E(x,y)) = 1$. Notice that 
\begin{equation*}
    \lambda(E(x,y)) = \mu_{b} \otimes \nu_{b}(E(x,y)) = \int_{S_{x}(Nb)} F_{\mu_{b}}^{E} d(\nu_{b}|_{N_b})
    \end{equation*}
Now, let $q \in \supp(\nu_{b}|_{Nb})$. Since $\nu(E(x,b)) = 1$, we know that $E(x,b) \in q$. So if $c \models q$, then $\models E(c,b)$. Moreover, since $E(x,y)$ is an equivalance relation, we conclude that $E(x,b)$ is equal to $E(x,c)$ \textit{as definable sets}. Therefore $F_{\mu_{b}}^{E}(q) = \mu(E(x,c)) = \mu(E(x,b)) = 1$. Hence
\begin{equation*}
    \int F_{\mu_{b}}^{E}d\nu_{b} = \int 1 d\nu_{b} = 1. 
\end{equation*}
We now show that $\lambda$ witnesses $\mu \geq_{\mathbb{E},N} \nu$. Fix $\omega \in \E(\lambda,\mu)$. First notice that for any $c \in \mathcal{U}$,
\begin{equation*}
    \pi_{y}(\omega)(E(y,c)) = \omega(E(y,c) \wedge x = x) = \omega(E(y,c) \wedge E(x,y)) 
\end{equation*}
\begin{equation*}
    = \omega(E(x,c) \wedge y = y) = \pi_{x}(\omega)(E(x,c)) = \mu(E(x,c)) = 0. 
\end{equation*}
We now use Lemma \ref{intersection} to show that $\pi_{y}(\omega) = \nu$. Consider
\begin{align*}
    \mathcal{B} := \Bigg\{ & \bigwedge_{j = 1}^{n} I_{j}^{\varepsilon_{1}(j)}(y) \wedge  \bigwedge_{i=1}^{m} E^{\varepsilon_{2}(i)}(y,c_i): n,m \in \mathbb{N}, c_i \in \mathcal{U}, \\ & \varepsilon_{1}:\{1,...,n\} \to \{0,1\}, \varepsilon_{2}: \{1,...,m\} \to \{0,1\} \Bigg\}, 
\end{align*}
where $I_{j}(y) \in \{y < b, b < y, b = y, b \neq y: b \in \mathcal{U}\}$. Let $\theta(y) \in \mathcal{B}$. Notice that if $\varepsilon_{2}(i) = 1$ for any $i \leq n$, then 
\begin{equation*}
    \pi_{y}(\omega)(\theta(y)) = \pi_{y}(\omega) \left( \bigwedge_{j=1}^{n} I_i^{\varepsilon_{1}(j)}(y) \wedge \bigwedge_{i=1}^{m} E^{\varepsilon_{2}(i)}(y,c_{i})\right) = 0
\end{equation*}
since $\pi_{y}(\omega)(E(y,c)) = 0$ for any $c \in \mathcal{U}$. Additionally, $\nu(\theta(y)) = 0$ by the same reasoning. Now assume that $\varepsilon_{2}(i) = 0$ for every $1 \leq i \leq n$. Then
\begin{equation*}
   \pi_{y}(\omega)(\theta(y)) =  \pi_{y}(\omega) \left( \bigwedge_{j=1}^{n} I^{\varepsilon_{1}(j)}_j(y) \wedge \bigwedge_{i=1}^{n} \neg E(y,c)\right) \overset{(*)}{=} \pi_{y}(\omega) \left( \bigwedge_{j=1}^{n} I_j^{\varepsilon_{1}(j)}(y) \right)
\end{equation*}
    
\begin{equation*}\overset{(\dagger)}{=} \nu\left( \bigwedge_{i=1}^{n} I^{\varepsilon_{1}(j)}_j(y) \right) \overset{(*)}{=} \nu\left( \bigwedge_{i=1}^{n} I^{\varepsilon_{1}(j)}_j(y) \wedge \bigwedge_{i=1}^{m}  \neg E(y,c_i) \right) = \nu(\theta(y)). 
\end{equation*}
Notice that $(\dagger)$ holds by the fact that $T|_{<}$ is distal, $\nu$ is generically stable over $N$, and $\pi_{y}(\omega)|_{N} = \pi_{y}(\omega|_N) = \pi_{y}(\lambda) = \nu|_{N}$ and so we may apply Lemma \ref{distal}. Notice that $(*)$ holds since $\pi_{y}(\omega)(\bigwedge_{i=1}^{m} \neg E(x,c_i)) = \nu(\bigwedge_{i=1}^{m} \neg E(x,c_i)) = 1$. 
\end{example}

\bibliographystyle{plain}
\bibliography{bibliography}

\end{document}